\documentclass[reqno]{amsart}
\numberwithin{equation}{section}

\usepackage[colorlinks,linkcolor={blue}]{hyperref}

\usepackage[alphabetic,nobysame]{amsrefs}
\usepackage{bm}
\usepackage{amssymb}
\usepackage{graphicx} 
\usepackage{amscd}
\usepackage{enumerate}
\usepackage[font=footnotesize]{caption}
\usepackage{dsfont}
\usepackage[colorinlistoftodos]{todonotes}
\usepackage{tikz-cd}

\definecolor{darkred}{rgb}{1,0,0} 
\definecolor{darkgreen}{rgb}{0,0.4,0}
\definecolor{darkblue}{rgb}{0,0,1}

\hypersetup{colorlinks,
linkcolor=darkblue,
filecolor=darkgreen,
urlcolor=darkred,
citecolor=darkgreen}

\makeatletter
\providecommand\@dotsep{5}
\makeatother

\definecolor{orange}{RGB}{253,85,0}

 \newcommand{\HS}[1]{H_{#1}^{S^1}}
 \newcommand{\nablat}{\nabla_t}
 \newcommand{\proj}{p}
 \newcommand{\sign}{\mathrm{sign}}
 \newcommand{\fix}{\mathrm{fix}}
 \newcommand{\crit}{\mathrm{crit}}
 \newcommand{\ev}{\mathrm{ev}}
 \newcommand{\im}{\mathrm{im}}
 \newcommand{\inj}{\mathrm{inj}}
 \newcommand{\ind}{\mathrm{ind}}
 \newcommand{\mul}{\mathrm{mul}}
 \newcommand{\N}{\mathds{N}}
 \newcommand{\Z}{\mathds{Z}}
 \newcommand{\R}{\mathds{R}}

 \newcommand{\C}{\mathds{C}}
 \newcommand{\Q}{\mathds{Q}}

 \newcommand{\Eu}{E^{u}}
 \newcommand{\Es}{E^{s}}

 \newcommand{\id}{\mathrm{id}}

 \DeclareMathOperator{\interior}{int}

 \DeclareMathOperator*{\ttoup}{\llongrightarrow} 
  
 \DeclareMathOperator*{\eembup}{\llonghookrightarrow}

\DeclareRobustCommand{\longhookrightarrow}{\lhook\joinrel\relbar\joinrel\rightarrow}
\DeclareRobustCommand{\llonghookrightarrow}{\lhook\joinrel\relbar\joinrel\relbar\joinrel\rightarrow}
\DeclareRobustCommand{\llongrightarrow}{\relbar\joinrel\relbar\joinrel\rightarrow}

 \theoremstyle{plain}
 \newtheorem{MainThm}{Theorem}

 \newtheorem{theorem}{Theorem}[section]
 \newtheorem{proposition}[theorem]{Proposition}
 \newtheorem{lemma}[theorem]{Lemma}
 \newtheorem{corollary}[theorem]{Corollary}
 
 \theoremstyle{definition}

 \newtheorem{remark}[theorem]{Remark}
 \newtheorem{example}[theorem]{Example}

\title{Klingenberg's theorem on Anosov geodesic flows} 

\author[R. Assouline]{Rotem Assouline}

\author[M. Mazzucchelli]{Marco Mazzucchelli}
\address{Sorbonne Université, Université Paris Cité, CNRS, IMJ-PRG\newline\indent F-75005 Paris, France}
\email{assouline@imj-prg.fr}
\email{marco.mazzucchelli@imj-prg.fr}

\thanks{Rotem Assouline is supported by the Fondation Sciences Mathématiques de Paris (FSMP) and the Rothschild Fellowship (Yad Hanadiv). Marco Mazzucchelli is partially supported by the ANR grant QuanQual (ANR-26-CE40-0676).}

\date{August 10, 2026}



\begin{document}

\begin{abstract}
In 1970, Klingenberg proved a fundamental result on closed Riemannian manifolds with Anosov geodesic flows, asserting that such manifolds are without conjugate points. The proof, based on the Morse theory of the energy functional, had crucial gaps pointed out by Anosov, who saved the theorem by providing a completely independent proof. In this article, we fill the gaps in the original argument, and thus provide a full Morse theoretic proof of Klingenberg's theorem. 
\tableofcontents
\end{abstract}

\maketitle

\section{Introduction}

The flow $\phi_t$ of a nowhere vanishing vector field $X$ on a closed connected manifold $N$ of dimension at least 3 is called \emph{Anosov} when its dynamics is hyperbolic, meaning that the tangent spaces admit splittings 
\[T_yN=\Es_y\oplus\Eu_y\oplus\langle X(y)\rangle,\qquad \forall y\in N,\] 
where $\Es_y$ and $\Eu_y$ are vector subspaces of positive dimension satisfying the following properties:
\begin{enumerate}[(i)]
\setlength{\itemsep}{3pt}
\item $d\phi_t(\Es)=\Es$ and $d\phi_t(\Eu)=\Eu$ for all $t\in\R$,

\item\label{i:Anosov2} $\|d\phi_t|_{\Es}\|\leq c e^{-t/c}$ and $\|d\phi_{-t}|_{\Eu}\|\leq c e^{-t/c}$ for all $t\geq0$, where $c\geq1$ is a constant independent of $t$, and $\|\cdot\|$ is a Riemannian norm on $N$.
\end{enumerate}
This class of dynamical systems was introduced by Anosov in his seminal work \cite{Anosov:1967aa}, where he also established its remarkable properties: ergodicity, a closing lemma, and structural stability, among others. Since then, Anosov flows have become of major interest in geometry and dynamics, see e.g.\ \cite{Katok:1995aa, Fisher:0aa}.

An \emph{Anosov Riemannian manifold} is a closed connected Riemannian manifold $(M,g)$ whose geodesic flow is Anosov. We recall that the geodesic flow $\phi_t$ is the flow on the unit tangent bundle $SM=\big\{v\in TM\ \big|\ \|v\|_g=1\big\}$ whose orbits have the form $\phi_t(\dot\gamma(0))=\dot\gamma(t)$, where $\gamma:\R\to M$ is any geodesic parametrized with unit speed $\|\dot\gamma\|_g\equiv1$.  
A major result in the above mentioned work \cite{Anosov:1967aa} asserts that closed connected Riemannian manifolds of negative sectional curvature are Anosov. While this curvature condition is sufficient to guarantee the hyperbolicity, it turns out to be not necessary: Eberlein \cite{Eberlein:1973aa} constructed an example of an Anosov Riemannian manifold whose sectional curvature vanishes on a non-empty open set, and thanks to the structural stability one can then perturb the Riemannian metric to obtain an Anosov Riemannian manifold whose sectional curvature attains positive values. Even more surprisingly, Donnay and Pugh \cite{Donnay:2003aa} constructed an Anosov Riemannian surface isometrically embedded in the Euclidean 3-space (and thus having somewhere positive Gaussian curvature).

While negative curvature is not a necessary property for the Anosov condition, it turns out that the geometry of a general Anosov Riemannian manifold is very similar to that of closed Riemannian manifolds of negative sectional curvature, due to the following result of Klingenberg \cite{Klingenberg:1974aa}. We recall that a closed Riemannian manifold $(M,g)$ is \emph{without conjugate points} when its geodesic flow $\phi_t$ satisfies \begin{align*}
 V(\phi_t(v))\cap d\phi_t(v)V(v) = \{0\},
 \qquad
 \forall v\in SM,\ t\in\R\setminus\{0\},
\end{align*}
where $V=\ker(dp)$ is the vertical distribution, and $p:SM\to M$ is the base projection of the unit tangent bundle.

\begin{MainThm}[Klingenberg]
\label{mt:Klingenberg}
On any Anosov Riemannian manifold, the following assertions hold.
\begin{enumerate}[$(i)$]
\setlength{\itemsep}{3pt}

\item\label{i:Klingenberg_1} No closed geodesic is contractible.

\item\label{i:Klingenberg_2} Every homotopy class of non-contractible loops contains a unique closed geodesic.

\item\label{i:Klingenberg_3} There are no conjugate points.

\end{enumerate}
\end{MainThm}

This fundamental theorem is of crucial importance in modern spectral rigidity theory, among other fields (see \cite{Guillarmou:2026aa,Lefeuvre:2025aa,Wilkinson:2026aa} and references therein). Klingenberg's proof was based on the Morse theory \cite{Milnor:1963aa} of the energy functional $E:\Lambda\to [0,\infty)$ on the free loop space $\Lambda=W^{1,2}(S^1,M)$.
The original article \cite{Klingenberg:1974aa}, as well as its rewriting in the later monograph \cite[Sec.~5.3]{Klingenberg:1978aa}, contained several inaccuracies, some of which were pointed out about a decade later by Anosov \cite{Anosov:1985aa}. The main issue stems from the invariance of the energy functional $E$ under the $S^1$ action $\tau\cdot\gamma:=\gamma(\tau+\cdot)$, where $\tau\in S^1$ and $\gamma\in\Lambda$. The space of non-trivial critical points $\crit^+(E):=\crit(E)\cap E^{-1}(0,\infty)$ consists of the 1-periodic closed geodesics $\gamma$ of the underlying Riemannian manifold. Any such $\gamma$ is never isolated in the critical set $\crit^+(E)$, since it belongs to a critical circle $S^1\cdot\gamma\subset\crit^+(E)$. In terms of Morse theory, the critical circle of a closed geodesic of Morse index $i$ and energy $c$ in an Anosov Riemannian manifold will generate a two-dimensional subspace in the rational homology of the pair $(\Lambda^{\leq c},\Lambda^{<c})$, namely a one-dimensional subspace in degree $i$ and another one-dimensional subspace in degree $i+1$. Here, $\Lambda^{\leq c}:=E^{-1}[0,c]$ and $\Lambda^{<c}:=E^{-1}[0,c)$. In his article, Klingenberg treats any closed geodesic of index $i$ as if it were an isolated critical point of the energy functional, thus generating only the one-dimensional subspace in degree $i$.

Fortunately, in the same article \cite{Anosov:1985aa}, Anosov saved Klingenberg's theorem by providing a new, beautiful, and totally independent proof of the result. He showed that the presence of a pair of conjugate points in an Anosov Riemannian manifold $(M,g)$ would allow to construct a closed hypersurface immersed in the unit tangent bundle $SM$ and transverse to all the orbits of the geodesic flow, something which is forbidden by the contact geometry of $SM$ and  Stokes' theorem. A few years later, Mañé \cite{Mane:1987aa} also provided a proof of Klingenberg's theorem along similar lines (it seems that Mañé was unaware of Anosov's paper, perhaps due to the delay in the appearance of the English translation of Anosov's original manuscript in Russian). Mañé even claimed that the same result holds if the Riemannian manifold is non-compact, complete, with a uniform lower bound on the sectional curvature, and under the suitable generalization of the Anosov condition in the non-compact setting. Actually, in the non-compact setting, Knieper \cite{Knieper:2002aa} found a gap in Mañé's proof  \cite[Prop.~II.2]{Mane:1987aa}, and provided a correct proof under some extra assumptions \cite{Knieper:2018aa}. A proof of the full Mañé theorem in the non-compact setting was given by Melo and Romaña in the preprint \cite{Melo:2020aa}.

In the present article, we provide a complete and rigorous Morse theoretic proof of Klingenberg's theorem, in the spirit of the original article \cite{Klingenberg:1974aa}. For this purpose, we work in the setting of $S^1$-equivariant Morse theory for the energy functional, first introduced by Hingston \cite{Hingston:1984aa} about a decade after Klingenberg's result. The details that we provide show that, despite the gaps, Klingenberg had the right approach for the proof of Theorem~\ref{mt:Klingenberg}. The outline of the argument is the following. 
The stable and unstable distributions $\Es$ and $\Eu$ of an Anosov Riemannian manifold force all closed geodesics to have even Morse indices. This prevents the existence of contractible closed geodesics of index 0; indeed, if a contractible closed geodesic $\gamma$ existed, one could set up a min-max procedure with the energy functional over the space of homotopies of loops going from $\gamma$ to a constant loop, and obtain another contractible closed geodesic of index 1. For the same reason, any homotopy class of non-contractible loops must contain a unique closed geodesic of index 0. In order to exclude the existence of contractible closed geodesics of positive index, we build on a beautiful idea that goes back to Fet \cite{Fet:1965aa}. Arguing by contradiction, we assume that there exist contractible closed geodesics, and consider one such $\gamma$ of minimal energy $c:=E(\gamma)>0$ and some index $i>0$. This gives two critical circles $S^1\cdot\gamma$ and $S^1\cdot\overline\gamma$, where $\overline\gamma=\gamma(-\cdot)$ is the same closed geodesic with reversed parametrization. For $\epsilon>0$ small enough, these two critical circles admit associated relative cycles $\Sigma$ and $\overline\Sigma$ spanning a two-dimensional subspace of the rational $S^1$-equivariant homology $\HS{i}(\Lambda^{<c+\epsilon},\Lambda^{<c})$, where $\Lambda^{<b}:=E^{-1}[0,b)$ for each $b>0$. It turns out that the relative cycle $\Sigma+\overline\Sigma$ vanishes in $\HS{i}(\Lambda^{<b},\Lambda^{<c})$ for a large enough $b>c$, and Morse theory implies that the infimum of the set of such values $b$ is the energy of a closed geodesic of index $i+1$, violating the allowed parity of the indices. Finally, an argument involving the homotopy long exact sequence of the fibration $\Lambda\to M$, $\gamma\mapsto\gamma(0)$ rules out the existence of non-contractible closed geodesics of positive index.

Unlike Anosov's and Mañé's arguments, the Morse theoretic one relies on the invariance of the energy functional $E$ under the orientation reversing $\Z_2$ action $\gamma\mapsto\overline\gamma$. While this is certainly a drawback, as it prevents the proof from being extended beyond the Riemannian case to Anosov geodesic flows of non-reversible Finsler metrics, we hope that the arguments help shed light on the still rather unexplored $\Z_2$ symmetry of the energy functional. Such symmetry may be the key for establishing the long-standing closed geodesics conjecture, which is known to fail in the non-reversible Finsler case: any closed Riemannian manifold of dimension at least two has infinitely many closed geodesics.

\subsection{Organization of the paper}
Given the troubled history of Klingenberg's theorem, we insisted on making this article self-contained.
In Section~\ref{s:preliminaries} we provide the background on the variational theory of closed geodesics that is needed for the proof of Theorem~\ref{mt:Klingenberg}. In Section~\ref{s:spherical_closed_geodesics} we introduce the notion of spherical closed geodesics, and provide a version of the original argument of Fet in the setting of $S^1$-equivariant Morse theory, asserting that, in a bumpy closed Riemannian manifold, the presence of spherical closed geodesics of index $i$ and energy $c$ forces the presence of a closed geodesic of index $i+1$ and energy above $c$. Finally, in Section~\ref{s:Anosov_closed_geodesics}, we provide the proof of Theorem~\ref{mt:Klingenberg}.

\subsection{Acknowledgments} 
The second author learned about Klingenberg's theorem from a conversation with Gerhard Knieper. We thank him for his interest in the present work and for his encouragement.

\section{Preliminaries}
\label{s:preliminaries}

The material in this section is standard, and can be mostly found in, e.g.\ \cite{Klingenberg:1978aa, Oancea:2015aa}.

\subsection{The free loop space}\label{s:setting}

Let $(M,g)$ be a closed connected Riemannian manifold.
We denote by $\Lambda:=W^{1,2}(S^1,M)$ its free loop space, where $S^1=\R/\Z$. The orthogonal group $O(2)=S^1\rtimes\Z_2$ acts on $S^1$, and therefore also on $\Lambda$ by
\begin{align*}
 (g\cdot\gamma)(t)=\gamma(g\cdot t),\qquad\forall g\in O(2),\ \gamma\in\Lambda,\ t\in S^1.
\end{align*}
More specifically, $S^1$ acts on $\Lambda$ by translations
\begin{align*}
\tau\cdot\gamma=\gamma(\tau+\cdot),
\qquad\forall\tau\in S^1,\ \gamma\in\Lambda,
\end{align*}
whereas $\Z_2=\{1,\rho\}$ acts on $\Lambda$ by time-reversal
\begin{align*}
 \rho\cdot\gamma=\overline\gamma:=\gamma(-\cdot),
 \qquad\forall\gamma\in\Lambda.
\end{align*}
We will mostly consider the actions of $S^1$ and $\Z_2$ separately.
We denote by $\Lambda^0$ the subspace of $\Lambda$ consisting of the constant loops. Notice that $\Lambda^0$ is diffeomorphic to $M$, and is the set of fixed points of the $O(2)$ action, or equivalently the set of fixed points of the $S^1$ action.

The multiplicity of a non-constant loop $\gamma\in\Lambda\setminus\Lambda^0$ is the positive integer
\begin{align*}
 \mul(\gamma)=\max\big\{ m\in\N=\{1,2,3,...\}\ \big|\ \gamma=\tfrac1m\cdot\gamma \big\}.
\end{align*}
The loop $\gamma$ is called \emph{primitive} when $\mul(\gamma)=1$. For each integer $m\geq1$, the $m$-th iterate of $\gamma\in\Lambda$ is the loop $\gamma^m=\gamma(m\,\cdot)\in\Lambda$, which has multiplicity 
\[\mul(\gamma^m)=m\,\mul(\gamma).\]

\subsection{The energy functional}

The energy functional $E:\Lambda\to [0,\infty)$ is given by 
\begin{align*}
E(\gamma)
=
\int_{S^1}
\|\dot\gamma(t)\|_g^2\,dt.
\end{align*}
The subspace of constant curves $\Lambda^0$ is precisely the subspace of global minimizers $E^{-1}(0)$. Aside from $\Lambda^0$, the remaining critical points
\[
\crit^+(E):=\crit(E)\setminus\Lambda^0=\crit(E)\cap E^{-1}(0,\infty)
\]
form the space of 1-periodic closed geodesics. We filter the free loop space $\Lambda$ by means of $E$, denoting
\begin{align*}
 \Lambda^{<c}:=E^{-1}[0,c),
 \qquad
 \Lambda^{\leq c}:=E^{-1}[0,c].
\end{align*}
Analogously, for each subset $U\subset\Lambda$, we denote
\begin{align*}
  U^{<c}:=\Lambda^{<c}\cap U,
 \qquad
 U^{\leq c}:=\Lambda^{\leq c}\cap U.
\end{align*}
The energy $E$ is $O(2)$-invariant, and any element in $\Lambda\setminus\Lambda^0=E^{-1}(0,\infty)$ has finite stabilizer. Actually, any closed geodesic $\gamma\in\crit^+(E)$ has finite stabilizer that is a subgroup of $S^1$.

\subsection{Non-degenerate closed geodesics}\label{ss:non_degenerate}

A closed geodesic $\gamma\in\crit^+(E)$ is called \emph{non-degenerate} when the kernel of the Hessian of $E$ at $\gamma$ is given by
\begin{align*}
\ker( d^2E(\gamma) ) = T_\gamma(S^1\cdot \gamma).
\end{align*}
Notice that, due to the $O(2)$-invariance of the energy functional $E$, if $\gamma$ is non-degenerate then any other point in its critical manifold $O(2)\cdot\gamma$ is non-degenerate as well.
A Riemannian manifold $(M,g)$ all of whose critical circles in $\crit^+(E)$ are non-degenerate is called \emph{bumpy}.

The Morse \emph{index} $\ind(\gamma)$ of a closed geodesic $\gamma\in\crit^+(E)$, which is always finite, is the maximal dimension of a vector subspace of $T_\gamma\Lambda$ over which $d^2E(\gamma)$ is negative-definite. If $\gamma$ is non-degenerate, the Morse-Bott lemma provides an $S^1$-invariant neighborhood of $S^1\cdot \gamma$ diffeomorphic to an open neighborhood $U$ of the zero-section in a Hilbert vector bundle $V=V^-\oplus V^+\to S^1\cdot\gamma$, where $V^-$ has rank $\ind(\gamma)$, while $V^+$ has infinite rank. In the local coordinates given by this diffeomorphism, the critical circle $S^1\cdot\gamma$ corresponds to the zero-section, and the energy $E$ is given by
\begin{align}
\label{e:Morse_Bott}
 E(v)=-\|v^-\|^2+\|v^+\|^2 + c,\qquad\forall v=(v^-,v^+)\in U,
\end{align}
where $c=E(\gamma)$.

\subsection{The anti-gradient flow}\label{ss:antigradient_flow}
The Riemannian metric $g$ induces an $O(2)$-invariant complete Riemannian metric $G$ on $\Lambda$ given by
\begin{align*}
 G(Y,Z)=\int_{S^1}\big( g(Y(t),Z(t))+g(\nablat Y,\nablat Z) \big)\,dt,
 \qquad\forall Y,Z\in T_\gamma\Lambda,
\end{align*}
where $\nablat$ denotes the Levi-Civita covariant derivative. We denote by $\nabla E$ the gradient of the energy functional with respect to $G$. While the free loop space $\Lambda$ is not compact (being infinite dimensional), it enjoys a weak compactness property, called the Palais-Smale condition, which is enough for variational arguments: any sequence $\gamma_n\in\Lambda$ with $E(\gamma_n)$ uniformly bounded and $\|\nabla E(\gamma_n)\|_G\to0$ admits a subsequence converging to a critical point of $E$.

The flow $\Psi_s$ of $-\nabla E$ is called the anti-gradient flow. Along the orbit going through any loop $\gamma\not\in\crit(E)$, the energy function $s\mapsto E(\Psi_s(\gamma))$ is strictly decreasing. Moreover, for each $c>0$ and for each open neighborhood $U\subset\Lambda$ of the set of critical points $\crit^+(E)\cap E^{-1}(c)$, there exists $\epsilon>0$ and $s>0$ such that
\begin{align*}
 \Psi_s(\Lambda^{<c+\epsilon})\subset \Lambda^{<c}\cup U.
\end{align*}

\subsection{The geodesic flow}

We consider the unit tangent bundle 
\[SM=\big\{v\in TM\ \big|\ \|v\|_g=1\big\},\] 
with its base projection $\proj:SM\to M$, 
and the geodesic flow $\phi_t:SM\to SM$, which is defined by 
\[\phi_t(v)=\dot\gamma_v(t),\] 
where $\gamma_v:\R\to M$ is the geodesic such that $\dot\gamma_v(0)=v$. We denote by $\lambda$ the canonical contact 1-form on $SM$, defined by 
\begin{align}
\label{e:contact_form}
\lambda_v=g(v,d\proj(v)\,\cdot\,).
\end{align}
Its kernel $\Xi=\ker(\lambda)\subset T(SM)$ is the canonical contact distribution of the unit tangent bundle. The geodesic vector field $X$ on $SM$, which is the infinitesimal generator of $\phi_t$, is the Reeb vector field of the contact form $\lambda$, i.e.
\[\lambda(X)\equiv1,\quad d\lambda(X,\cdot)\equiv0.\]
The geodesic flow satisfies $\phi_t^*\lambda=\lambda$, and in particular $d\phi_t(v)\Xi_v=\Xi_{\phi_t(v)}$. 

Another relevant distribution is the \emph{vertical} one
\begin{align*}
 V:=\ker(d\proj).
\end{align*}
For each $v\in SM$ and $t>0$ sufficiently small, we have a trivial intersection \[V(\phi_t(v))\cap d\phi_t(v)V(v)=\{0\}.\] 
If there exist larger values $t>0$ such that $V(\phi_t(v))\cap d\phi_t(v)V(v)\neq\{0\}$, the points $\gamma_v(0)$ and $\gamma_v(t)$ are said to be \emph{conjugate} along the geodesic path $\gamma_v:[0,t]\to M$. The classical Morse index theorem from Riemannian geometry \cite[Sec.~15]{Milnor:1963aa} implies that the index of any closed geodesic $\gamma\in\crit^+(E)$ is bounded from below by the number of time values $t\in(0,1)$ such that $\gamma(0)$ and $\gamma(t)$ are conjugate along $\gamma|_{[0,t]}$, counted with a suitable multiplicity, i.e.
\begin{align}
\label{e:index_bound_conjugate_points}
 \ind(\gamma)
 \geq
 \sum_{t\in(0,1)}
 \dim\big(V(\phi_{\tau t}(v))\cap d\phi_{\tau t}(v)V(v)\big),
\end{align}
where $\tau=E(\gamma)^{1/2}$ and $v=\dot\gamma(0)/\|\dot\gamma(0)\|_g$.

\subsection{The Morse index of hyperbolic closed geodesics}

Let $\gamma\in\crit^+(E)\cap E^{-1}(\tau^2)$ be a closed geodesic, and $v:=\dot\gamma(0)/\|\dot\gamma(0)\|_g\in SM$ its initial oriented unit tangent vector, so that $t\mapsto\phi_t(v)$ is the corresponding $\tau$-periodic orbit of the geodesic flow, i.e. 
\[
\phi_{\tau t}(v) = \dot\gamma(t)/ \|\dot\gamma(t)\|_g.
\]
The kernel of the Hessian $d^2E(\gamma)$ can be characterized in terms of the linearized geodesic flow: there is an isomorphism
\begin{align*}
\ker(d\phi_{\tau}(v)-I) & \overset{\cong}{\longrightarrow}\ker(d^2E(\gamma)),\\
w & \longmapsto J_w,
\end{align*}
where $J_w\in T_\gamma\Lambda$ is the 1-periodic Jacobi field along $\gamma$ defined by
\[
J_w(t)=d(\proj\circ\phi_{\tau t})(v)w.
\]
For $w=X(v)$, the corresponding Jacobi field is $J_w=\dot\gamma$. The closed geodesic $\gamma$ is non-degenerate if and only if $\ker(d\phi_{\tau}(v)-I)=\langle X(v)\rangle$, namely if and only if the restriction $d\phi_\tau(v)|_{\Xi_v}$ does not have the eigenvalue 1.

The closed geodesic $\gamma$ is said to be \emph{hyperbolic} when $d\phi_{\tau}(v)|_{\Xi_v}$ has no eigenvalues on the unit circle of the complex plane (in particular, the critical circle $S^1\cdot\gamma$ is non-degenerate). Under this condition, we have a splitting
\begin{align*}
 \Xi_v = \Es_v \oplus \Eu_v,
\end{align*}
where $\Es_v$ (resp.\ $\Eu_v$) is the direct sum of the generalized eigenspaces of $d\phi_{\tau}(v)|_{\Xi_v}$ corresponding to the eigenvalues of modulus less than 1 (resp.\ larger than 1). 
Since $d\phi_{\tau}(v)|_{\Xi_v}$ is a linear symplectic automorphism of the symplectic vector space $(\Xi_v,d\lambda_v)$, the subspaces $\Es_v$ and $\Eu_v$ have the same dimension
\begin{align}
\label{e:Es_Eu_same_dim}
 \dim(\Es_v)=\dim(\Eu_v)=\dim(M)-1.
\end{align}
For each $t\in\R/\tau\Z$, we set
\begin{align*}
\Es_{\phi_t(v)}=d\phi_t(v)\Es_v,
\qquad
\Eu_{\phi_t(v)}=d\phi_t(v)\Eu_v.
\end{align*}
The union of these vector spaces for varying $t$ form the stable bundle $\Es$ and the unstable bundle $\Eu$ over the $\tau$-periodic orbit $t\mapsto\phi_t(v)$, or equivalently over the circle $\R/\tau\Z$. By a theorem of Klingenberg \cite[Prop.~5]{Klingenberg:1974aa}, which is a hyperbolic version of the aforementioned Morse index theorem from Riemannian geometry \cite{Milnor:1963aa}, the index of $\gamma$ can be computed in terms of the stable bundle $\Es$ and of the vertical bundle $V$ as
\begin{align}
\label{e:index_hyperbolic}
 \ind(\gamma) = \sum_{t\in\R/\tau\Z} d_t,
\end{align}
where
\begin{align*}
d_t:=\dim\big(\Es_{\phi_t(v)}\cap V_{\phi_t(v)}\big).
\end{align*}
This formula has the following important consequence, first remarked by Klingenberg in the same article. We denote by $N\gamma\subset TM|_{\gamma}$ the normal bundle of the closed geodesic $\gamma$ seen as an immersed circle in $M$.

\begin{proposition}\label{p:index_parity}
The Morse index $\ind(\gamma)$ is even if and only if $\Es$ and $N\gamma$ are either both orientable or both unorientable vector bundles over the circle.
\end{proposition}

\begin{proof}
The linearized base projection $d\proj$ maps $\Es$ to $N\gamma$, and we denote its fiber restriction as
\begin{align*}
\Pi_t:=d\proj(\phi_t(v))|_{\Es_{\phi_t(v)}} :\Es_{\phi_t(v)}\to N_{\gamma(t/\tau)}\gamma,\qquad\forall t\in\R.
\end{align*}
Let $t_0\in[0,\tau)$ be any value such that $d_{t_0}=0$. 
Since 
\[\ker(\Pi_t)=\Es_{\phi_t(v)}\cap V_{\phi_t(v)},
\qquad
\dim(\Es_{\phi_t(v)})=\dim(N_{\gamma(t/\tau)}\gamma),\] 
the linear map $\Pi_{t_0}$ is an isomorphism. We fix an arbitrary orientation $o_{t_0}$ of $N_{\gamma(t_0/\tau)}\gamma$, and extend it continuously to a family of orientations $o_{t}$ of $N_{\gamma(t/\tau)}\gamma$ for all $t\in\R$. Notice that $o_{t_0}=o_{t_0+\tau}$ if and only if the normal bundle $N\gamma$ is orientable. Next, we fix an arbitrary orientation $u_{t_0}$ of $\Es_{\phi_{t_0}(v)}$, and extend it  continuously to a family of orientations $u_{t}$ of $\Es_{\phi_{t}(v)}$ for all $t\in\R$. Once again, $u_{t_0}=u_{t_0+\tau}$ if and only if the stable bundle $\Es$ is orientable.
We define $\sign(\Pi_t)$ to be the sign of the determinant of $\Pi_t$ with respect to any choice of oriented bases of $\Es_{\phi_t(v)}$ and $N_{\gamma(t/\tau)}\gamma$. Namely, 
\begin{align*}
\sign(\Pi_t)
=
\begin{cases}
 1,& \mbox{if $\Pi_t$ is orientation preserving},\\
 0,& \mbox{if $\Pi_t$ is non-invertible},\\
 -1,& \mbox{if $\Pi_t$ is orientation reversing}.
\end{cases}
\end{align*}
At each $t\in(t_0,t_0+\tau)$ such that $d_t\neq0$ we have that $\sign(\Pi_{t^-})=\sign(\Pi_{t^+})$ if and only if $d_t$ is even. Therefore, by the index formula~\eqref{e:index_hyperbolic}, we have that $\sign(\Pi_{t_0})=\sign(\Pi_{t_0+\tau})$ if and only if $\ind(\gamma)$ is even. Finally, $\sign(\Pi_{t_0})=\sign(\Pi_{t_0+\tau})$ if and only if $\Es$ and $N\gamma$ are either both orientable or both unorientable vector bundles.
\end{proof}

For any closed geodesic $\gamma\in\crit^+(E)$, the function $m\mapsto\ind(\gamma^m)$ is either identically zero or has linear growth. If $\gamma$ is hyperbolic, the index formula~\eqref{e:index_hyperbolic} actually implies that the function $m\mapsto\ind(\gamma^m)$ is exactly linear, i.e.
\[
\ind(\gamma^m)=m\,\ind(\gamma).
\]

On any Anosov Riemannian manifold, any closed geodesic $\gamma\in\crit^+(E)$ is hyperbolic. If $v:=\dot\gamma(0)/\|\dot\gamma(0)\|_g$, the stable and unstable bundles of the associated periodic orbit $t\mapsto\phi_t(v)$ of the geodesic flow are the restrictions of the ambient stable and unstable bundles $\Es,\Eu\subset T(SM)$ respectively. The Anosov condition implies that $\Es_v$ and $\Eu_v$ depend continuously (and even H\"older continuously) on $v\in SM$, and in particular their dimension is independent of $v$. Therefore 
\[\dim(\Es_v)=\dim(\Eu_v),\qquad \forall v\in SM,\] 
since this identity holds when $v$ lies on a closed orbit of the geodesic flow $\phi_t$, as was recalled in \eqref{e:Es_Eu_same_dim}.

\section{Spherical closed geodesics}
\label{s:spherical_closed_geodesics}

\subsection{Spherical maps seen in the free loop space}\label{ss:spherical}
Let $M$ be a connected closed manifold.
While the Sobolev free loop space $\Lambda=W^{1,2}(S^1,M)$ is well suited for the variational theory of the energy functional, in this subsection we also employ the continuous free loop space 
\[\Upsilon:=C^0(S^1,M),\] 
endowed with the $C^0$-topology and equipped with the action of $O(2)=S^1\rtimes\Z_2$ analogous to the one on $\Lambda$. We again set $\Upsilon^0$ to be the subspace of $\Upsilon$ consisting of the constant loops. Namely, $\Upsilon^0$ is the subset of fixed points of the whole $O(2)$ action, or equivalently of the whole $S^1$ action. The Sobolev embedding theorem guarantees a continuous inclusion $\Lambda\hookrightarrow\Upsilon$, and a standard regularization argument implies that this inclusion is a homotopy equivalence. Since $\Upsilon^0=\Lambda^0$, we actually have that the inclusion of pairs $(\Lambda,\Lambda^0)\hookrightarrow(\Upsilon,\Upsilon^0)$ is a homotopy equivalence.

For a positive integer $n\geq1$, consider the unit $n$-sphere $S^n$ as the equator of the unit $(n+1)$-sphere $S^{n+1}$. We write the ambient Euclidean space $\R^{n+2}$ as $\R^n\times\C$, so that we can use the complex notation in the last two coordinates. We equip this Euclidean space with the $S^1$ action 
\begin{align*}
 t\cdot(y,z)=(y,e^{i2\pi t}z),\qquad\forall t\in S^1,\ (y,z)\in \R^n\times\C,
\end{align*}
which preserves $S^{n+1}$. We identify the compact $n$-ball $B^n$ with the lower hemisphere of $S^n$, i.e.
\begin{align}
\label{eq:Bndef}
 B^n=\big\{(y,w,0)\in S^n\subset\R^n\times\R\times\{0\}\ \big|\ w\leq 0\big\}.
\end{align}
We have a disjoint union decomposition 
\[S^{n+1}=\partial B^n \sqcup S^1\cdot\interior(B^n).\]
The circle $S^1$ acts trivially on $\partial B^n$, whereas any point of $\interior(B^n)$ has trivial stabilizer for the $S^1$ action.

\begin{figure}
\begin{center}
\includegraphics{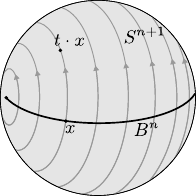}
\caption{The singular foliation that decomposes $S^{n+1}$ as a family of loops parametrized by the lower hemisphere $B^n$ of the equatorial sphere $S^n$.}
\label{f:sphere_decomposition}
\end{center}
\end{figure}

Any continuous map $\sigma:(B^n,\partial B^n)\to(\Upsilon,\Upsilon^0)$ can be equivalently described as a continuous map $\sigma_0:S^{n+1}\to M$, under the identity
\begin{align*}
\sigma_0(y,e^{i2\pi t}z):=\sigma(y,z)(t),\qquad\forall (y,z)\in B^n,\ t\in S^1,
\end{align*}
see Figure~\ref{f:sphere_decomposition}. This correspondence gives a homeomorphism
\begin{equation}\label{e:spherical_maps}
\begin{split}
C^0((B^n,\partial B^n),(\Upsilon,\Upsilon^0)) &\overset{\cong}{\longrightarrow} C^0(S^{n+1}, M),\\
\sigma & \longmapsto \sigma_0. 
\end{split}
\end{equation}
While the restriction of this map to $C^0((B^n,\partial B^n),(\Lambda,\Lambda^0))$ is not a homeomorphism, we have at least the following.

\begin{lemma}\label{l:spherical_maps_Lambda}
The map
\begin{align*}
C^0((B^n,\partial B^n),(\Lambda,\Lambda^0)) & \longrightarrow C^0(S^{n+1}, M),\\
\sigma & \longmapsto \sigma_0
\end{align*}
is a homotopy equivalence.
\end{lemma}

\begin{proof}
Let $\psi:(\Upsilon,\Upsilon^0)\to(\Lambda,\Lambda^0)$ be a homotopy inverse of the inclusion $(\Lambda,\Lambda^0)\hookrightarrow(\Upsilon,\Upsilon^0)$. This induces a continuous map 
\begin{align*}
\Psi:C^0((B^n,\partial B^n),(\Upsilon,\Upsilon^0)) \to C^0((B^n,\partial B^n),(\Lambda,\Lambda^0)) 
\end{align*}
given by $\Psi(\sigma)(x)=\psi(\sigma(x))$, which is a homotopy inverse of the inclusion 
\[
C^0((B^n,\partial B^n),(\Lambda,\Lambda^0))\longhookrightarrow C^0((B^n,\partial B^n),(\Upsilon,\Upsilon^0)).
\]
By composing this latter inclusion with the homeomorphism \eqref{e:spherical_maps}, we infer the lemma.
\end{proof}

\begin{corollary}\label{c:spherical_maps}
If the homotopy group $\pi_{n+1}(M)$ is non-trivial, then there exists $\sigma:(B^n,\partial B^n)\to(\Lambda,\Lambda^0)$ that is not homotopic as map of pairs to a map with image contained in $\Lambda^0$.
\end{corollary}

\begin{proof}
The non-vanishing of $\pi_{n+1}(M)$ implies that there exists a continuous map $\sigma_0:S^{n+1}\to M$ that is not homotopic to a constant. By Lemma~\ref{l:spherical_maps_Lambda}, up to replacing $\sigma_0$ with a homotopic map, we can assume that $\sigma_0$ is the image of some $\sigma\in C^0((B^n,\partial B^n),(\Lambda,\Lambda^0))$ under the map \eqref{e:spherical_maps}. If $\sigma$ were homotopic as map of pairs to a map with image contained in $\Lambda^0$, then it would also be homotopic as map of pairs to a constant map identically equal to a constant loop $\gamma\equiv y$. But this would imply that $\sigma_0$ is homotopic to the constant map identically equal to $y$.
\end{proof}

\begin{remark}\label{r:based_loop_space}
In topology, a version of Lemma~\ref{l:spherical_maps_Lambda} is usually stated for the based loop space
\begin{align*}
 \Omega_x:=\big\{ \gamma\in\Lambda\ |\ \gamma(0)=x \big\},
\end{align*}
where $x\in M$ is a basepoint. For each $\gamma\in\Omega_x$, a map analogous to \eqref{e:spherical_maps} induces an isomorphism of homotopy groups
\begin{align*}
 \pi_n(\Omega_x,\gamma) \overset{\cong}{\longrightarrow} \pi_{n+1}(M,x),
 \qquad\forall n\geq0.
\end{align*}
Since $M$ is connected, the homotopy type of $\Omega_x$ is independent of the choice of the basepoint $x$, which for this reason is occasionally omitted from the notation and from the homotopy groups.
The spaces $\Omega_x$, $\Lambda$, and $M$ are related by the fibration 
\[\ev:\Lambda\to M,\qquad\ev(\gamma)=\gamma(0),\]
whose fibers are $\ev^{-1}(x)=\Omega_x$. This fibration induces the long exact sequence of homotopy groups
\begin{align*}
 ...
 \overset{\ev_*\ }{\longrightarrow}
 \pi_{n+1}(M)
 \overset{\partial_*}{\longrightarrow}
 \pi_{n}(\Omega_x)
 \overset{\iota_*}{\longrightarrow}
 \pi_{n}(\Lambda)
 \overset{\ev_*\ }{\longrightarrow}
 \pi_{n}(M)
 \overset{\partial_*}{\longrightarrow}
 ...
 \overset{\iota_*}{\longrightarrow}
 \pi_{0}(\Lambda)
 \overset{\ev_*\ }{\longrightarrow}
 \pi_{0}(M)
\end{align*}
where $\iota_*$ is induced by the inclusion and $\partial_*$ is a connecting homomorphism. 
\end{remark}

\subsection{Descending manifolds}\label{ss:descending_manifolds}

In the following, we denote by $\HS*(\cdot)$ the $S^1$-e\-qui\-var\-i\-ant homology with coefficients in the field of rational numbers $\Q$. We recall that, for a topological space $Y$ equipped with an $S^1$ action, the Borel quotient is defined by $Y_{S^1}:=(Y\times ES^1)/S^1$. Here, $ES^1$ can be any contractible space equipped with a free $S^1$ action (for instance the unit sphere in an infinite-dimensional complex Hilbert space), and the circle $S^1$ acts diagonally on the product $Y\times ES^1$. The $S^1$-equivariant homology of $Y$ is defined by $\HS*(Y):=H_*(Y_{S^1})$, where $H_*(\cdot)$ is the ordinary singular homology with coefficients in $\Q$. If $S^1$ acts freely on $Y$, the quotient projection $Y_{S^1}\to Y/S^1$ is a homotopy equivalence. In particular, considering the canonical action of $S^1$ on itself, the Borel quotient $S^1_{S^1}$ is homotopy equivalent to the singleton $S^1/S^1$, namely $S^1_{S^1}$ is contractible.

Let $\gamma\in\crit^+(E)$ be a non-degenerate closed geodesic of energy $E(\gamma)=c$ and positive index $n=\ind(\gamma)>0$. A \emph{descending manifold} for $\gamma$ is a continuous map $\sigma:B^n\to\Lambda$ such that 
\begin{enumerate}[(i)]
\setlength{\itemsep}{3pt}
 
 \item $\sigma(x_0)=\gamma$ for some $x_0\in\interior(B^n)$,

 \item\label{i:desc2} $\sigma$ is a smooth embedding near $x_0$, and the Hessian $d^2E(\gamma)$ is negative definite on the image of $d\sigma(x_0)$, i.e.
 \[
 d^2E(\gamma)[d\sigma(x_0)v,d\sigma(x_0)v]<0,\quad\forall v\in T_{x_0}B^n\setminus\{0\},
 \]

 \item $\sigma(B^n\setminus\{x_0\})\subset\Lambda^{<c}$.
 
\end{enumerate}

\begin{example}
In the coordinates given by Morse-Bott lemma (Section~\ref{ss:non_degenerate}), an example of descending manifold is given by the inclusion of any compact $n$-ball in $V^-$ containing the origin in its interior.
\end{example}

\begin{remark}
Let $\rho$ be the generator of $\Z_2$, which acts on the free loop space $\Lambda$ by reversing the parametrization of every loop. If $\sigma$ is a descending manifold of $\gamma$, then $\overline\sigma:=\rho\cdot\sigma$ is a descending manifold of the reversed closed geodesic $\overline\gamma:=\rho\cdot\gamma$.
\end{remark}

Consider the product 
\[W:=B^n\times S^1,\] 
equipped with the $S^1$ action 
\[\tau\cdot(x,t):=(x,t+\tau),
\qquad\forall (x,t)\in W,\ \tau\in S^1.\] 
Its Borel quotient $W_{S^1}=B^n\times S^1_{S^1}$ is homotopy equivalent to $B^n$, and analogously $\partial W_{S^1}=\partial B^n\times S^1_{S^1}$ is homotopy equivalent to $\partial B^n$. Therefore
\begin{align}\label{e:HSW_HB}
\HS{p}(W,\partial W)
\cong
H_{p}(B^n,\partial B^n)
\cong
\begin{cases}
\Q,&\mbox{if }p=n,\\
0,&\mbox{otherwise}.
\end{cases}
\end{align}
Assume that $\gamma$ is primitive, i.e.
\begin{align*}
 \mul(\gamma)=1.
\end{align*}
We extend $\sigma$ to an $S^1$-equivariant continuous map
\begin{align*}
 \Sigma:W\to\Lambda,\qquad\Sigma(x,\tau)=\tau\cdot\sigma(x),
\end{align*}
and consider the image $\Sigma(W)$. The circle $S^1$ might not act freely on the whole $\Sigma(W)$, but it certainly acts freely at least on some $S^1$-invariant tubular neighborhood $U\subset\Lambda$ of the critical circle $S^1\cdot\gamma$, since $\gamma$ is primitive. Such a $U$ is homeomorphic to a neighborhood of the zero-section in a vector bundle $V=V^-\oplus V^+$ over the critical manifold $S^1\cdot\gamma$, and the Morse-Bott lemma allows to write $E|_U$ as a fiberwise quadratic form \eqref{e:Morse_Bott}. Let $Z\subset W$ be an $S^1$-invariant tubular neighborhood of $C:=\{x_0\}\times S^1$ that is small enough so that $\Sigma$ restricts to a smooth embedding $\Sigma|_Z:Z\hookrightarrow U$. Notice that $\Sigma(Z)$ has the structure of a trivial fiber bundle over the critical manifold $S^1\cdot\gamma$, with the zero-section $\Sigma(C)=S^1\cdot\gamma$ and complement $\Sigma(Z\setminus C)\subset\Lambda^{<c}$. Moreover, by property (\ref{i:desc2}) in the definition of descending manifold, $\Sigma(Z)$ is a graph over $V^-$.
Therefore $\Sigma$ induces the $S^1$-equivariant homology isomorphism
\begin{align*}
 \HS*(Z,\partial Z)
 \ttoup^{\Sigma_*}_{\cong}
 \HS*(\Lambda^{<c}\cup S^1\cdot\gamma,\Lambda^{<c}).
\end{align*}
This, together with $\Sigma(W\setminus C)\subset\Lambda^{<c}$, implies that  $\Sigma$ induces the $S^1$-equivariant homology isomorphism
\begin{align*}
 \HS*(W,\partial W)
 \ttoup^{\Sigma_*}_{\cong}
 \HS*(\Lambda^{<c}\cup S^1\cdot\gamma,\Lambda^{<c}).
\end{align*}
The inclusion
\begin{align*}
 (\Lambda^{<c}\cup \bigcup_{S^1\cdot\zeta} S^1\cdot\zeta,\Lambda^{<c})
 \longhookrightarrow
 (\Lambda^{\leq c},\Lambda^{<c}),
\end{align*}
with the union ranging over all critical manifolds $S^1\cdot\zeta\subset\crit^+(E)\cap E^{-1}(c)$, is a homotopy equivalence whose homotopy inverse is the time-$s$ map $\Psi_s$ of the anti-gradient flow $E$, for any $s>0$ (see Section~\ref{ss:antigradient_flow}). This, together with the excision property of singular homology, implies that $\Sigma$ induces the $S^1$-equivariant homology monomorphism
\begin{align*}
 \HS*(W,\partial W)
 \eembup^{\Sigma_*}
 \HS*(\Lambda^{\leq c},\Lambda^{<c}).
\end{align*}

\subsection{Spherical closed geodesics}
We say that a closed geodesic $\gamma$ is \emph{spherical} when it is non-degenerate, has positive index $n=\ind(\gamma)>0$, and admits a descending manifold $\sigma:B^n\to\Lambda$ such that $\sigma(\partial B^n)\subset\Lambda^0$. Notice that this implies that $\gamma$ is a contractible loop in $M$, and the whole $\sigma$ has image in the connected component $\Delta$ of $\Lambda$ consisting of the contractible loops. We refer to such a $\sigma$ as a \emph{spherical descending manifold}. The reason for this terminology is that $\sigma$ can be described as a map $\sigma_0:S^{n+1}\to M$ as in Lemma~\ref{l:spherical_maps_Lambda}. The following lemma provides a sufficient condition for a closed geodesic to be spherical.

\begin{lemma}\label{l:shortest_contractible}
A non-degenerate contractible closed geodesic of positive index and minimal energy among all contractible closed geodesics is spherical.
\end{lemma}

\begin{proof}
Let $\gamma\in\crit^+(E)\cap\Delta$ be a non-degenerate contractible closed geodesic of positive index $\ind(\gamma)>0$ and such that $c:=E(\gamma)\leq E(\zeta)$ for any other contractible closed geodesic $\zeta\in\crit^+(E)\cap\Delta$. Let $\Psi_s$ be the anti-gradient flow of the energy functional $E$, and $\sigma:B^n\to\Delta$ a descending manifold for $\gamma$. For each $s>0$, $\Psi_s\circ\sigma$ is still a descending manifold for  $\gamma$. Since there are no contractible closed geodesics of energy in $(0,c)$, for each $\epsilon>0$ there exists $s=s(\epsilon)>0$ large enough so that
\begin{align*}
\max_{\partial B^n} E\circ\Psi_s\circ\sigma < \epsilon.
\end{align*}
We require $\epsilon$ to be smaller than the squared injectivity radius $\inj(M,g)^2$. This implies that, for every boundary point $x\in\partial B^n$, the corresponding loop 
\[\zeta_x:=\Psi_s\circ\sigma(x)\] 
has length smaller than the injectivity radius $\inj(M,g)$. Therefore, we can build a homotopy $\zeta_{x,r}\in\Delta$, for $r\in[0,1]$, by
\begin{align}
\label{e:zeta_r}
 \zeta_{x,r}(t)=\exp_{\zeta_x(0)}((1-r)\exp_{\zeta_x(0)}^{-1}(\zeta_x(t))),
\end{align}
where $\exp$ denotes the exponential map of the Riemannian manifold $(M,g)$. This homotopy starts at $\zeta_{x,0}=\zeta_x$ and ends at the constant loop $\zeta_{x,1}\equiv\zeta_x(0)$. Moreover, if we choose $\epsilon$ small enough, we have $E(\zeta_{x,r})<E(\gamma)$ for all $x\in\partial B^n$ and $r\in[0,1]$.

Now, let us realize the compact ball $B^n$ as the unit ball in $\R^n$, and let the origin $x_0=0\in B^n$ be the special point mapped to $\sigma(0)=\gamma$. We build a spherical descending manifold $\nu:B^n\to\Delta$ for $\gamma$ as follows. For each $x\in B^n$, we set
\begin{align*}
 \nu(\tfrac12 x):=\Psi_s\circ\sigma(x).
\end{align*}
For each $x\in\partial B^n$ and $r\in[0,1]$, we set
\[
\nu(\tfrac{r+1}2 x):=\zeta_{x,r}.
\qedhere
\]
\end{proof}

Assume now that $\gamma$ is primitive and spherical, and let $\sigma$ be a spherical descending manifold of $\gamma$.
We denote by $\overline\sigma:=\rho\cdot\sigma$ its image under the generator of the $\Z_2$ action on $\Lambda$, and by $\Sigma$ and $\overline\Sigma$ the corresponding $S^1$-equivariant extensions (see Section~\ref{ss:descending_manifolds}). Since $S^1\cdot\gamma$ and $S^1\cdot\overline\gamma$ are disjoint critical circles in the connected components $\Delta$ of contractible loops, the maps $\Sigma$ and $\overline\Sigma$ induce homomorphisms in $S^1$-equivariant relative homology whose images span a 2-dimensional vector subspace
\[\langle\im(\Sigma_*),\im(\overline\Sigma_*)\rangle\subseteq\HS*(\Delta^{\leq c},\Delta^{<c}).\]
If instead we see $\Sigma_*$ and $\overline\Sigma_*$ as homomorphisms towards $\HS*(\Delta,\Delta^{0})$, with $\Delta^0:=\Lambda^0$ being the subspace of constant loops, we have the following.

\begin{lemma}\label{l:Fet}
The maps $\Sigma$ and $\overline\Sigma$ induce opposite homomorphisms
\begin{align}
\label{e:same_homomorphism}
\Sigma_*=-\overline\Sigma_*
:
\HS*(W,\partial W)
\to
\HS*(\Delta,\Delta^{0}).
\end{align}
\end{lemma}

\begin{proof}
As in Section~\ref{ss:spherical}, we consider the unit $n$-sphere $S^n$ as the equator of the unit $(n+1)$-sphere $S^{n+1}$, and the compact $n$-ball $B^n$ as the lower hemisphere of $S^n$ (see \eqref{eq:Bndef}). We denote by $r_\theta:S^{n+1}\to S^{n+1}$ the rotation by angle $\theta$ in the $(x_1,x_{n+2})$-plane, i.e.
\begin{align*}
r_{\theta}(x_1,...,x_{n+2})
=
(\cos(\theta)x_1-\sin(\theta)x_{n+2},x_2,...,x_{n+1},\cos(\theta)x_{n+2}+\sin(\theta)x_{1}).
\end{align*}
Notice that $r_0=\id$, and $r_\pi$ is an involution that preserves $B^n$ and reverses its orientation.

We see the spherical descending manifold $\sigma$ as a continuous map $\sigma_0:S^{n+1}\to M$ via the homeomorphism \eqref{e:spherical_maps}. We define the continuous homotopy 
$h_\theta:=\sigma_0\circ r_\theta$, which satisfies $h_0=\sigma_0$ and $h_\pi=\sigma_0\circ r_\pi=:\psi_0$. The continuous map of pairs $\psi:(B^n,\partial B^n)\to(\Delta,\Delta^0)$ associated with $\psi_0$ via the homeomorphism \eqref{e:spherical_maps} is given by
\begin{align*}
 \psi=\overline\sigma\circ r_\pi|_{B^n}.
\end{align*}
By Lemma~\ref{l:spherical_maps_Lambda}, $\sigma$ and $\psi$ are homotopic as maps of pairs. We consider the $S^1$-equivariant extension of $\psi$, which is the map
\[\Psi:(W,\partial W)\to(\Delta,\Delta^0),\qquad \Psi(x,\tau)=\tau\cdot\psi(x).\]
The $S^1$-equivariant extension of any homotopy from $\sigma$ to $\psi$ is a homotopy from $\Sigma$ to $\Psi$, and its existence implies that $\Sigma$ and $\Psi$ induce the same homomorphism
\begin{align}\label{e:Sigmas_Psis}
\Sigma_*=\Psi_*:
\HS*(W,\partial W)
\to
\HS*(\Delta,\Delta^{0}).
\end{align}
Notice that $\Psi=\overline\Sigma\circ R$, where $R:W\to W$ is the $S^1$-equivariant diffeomorphism 
$R(x,t)=(r_\pi(x),t)$.
Since $r_\pi$ reverses the orientation of $B^n$, it induces the homology isomorphism
\begin{align*}
r_{\pi*}=-\id_*:H_n(B^n,\partial B^n)\to H_n(B^n,\partial B^n).
\end{align*}
This, together with the isomorphism \eqref{e:HSW_HB}, implies
\begin{align*}
R_*=-\id_*:\HS{n}(W,\partial W)\to \HS{n}(W,\partial W).
\end{align*}
Therefore
\begin{align*}
\Psi_*=\overline\Sigma_*\circ R_*=-\overline\Sigma_*
:\HS*(W,\partial W)
\to
\HS*(\Delta,\Delta^{0}).
\end{align*}
This, together with~\eqref{e:Sigmas_Psis}, implies~\eqref{e:same_homomorphism}.
\end{proof}

Spherical primitive closed geodesics of index $n$ force the existence of further critical circles of contractible closed geodesics with $S^1$-equivariant local homology in degree $n+1$, as asserted by the following theorem, which essentially goes back to Fet \cite{Fet:1965aa}. For simplicity, we provide the statement for bumpy closed Riemannian manifolds (see Section~\ref{ss:non_degenerate}), but one could obtain a slightly weaker assertion by replacing the bumpy condition with the requirement that $\crit^+(E)\cap\Delta$ consists of isolated critical circles.

\begin{theorem}\label{t:Fet}
If $(M,g)$ is a bumpy closed Riemannian manifold, then for any spherical primitive closed geodesic $\gamma\in\crit^+(E)$ there exists a contractible closed geodesic $\zeta\in\crit^+(E)$ of energy $E(\zeta)>E(\gamma)$ and index $\ind(\zeta)=\ind(\gamma)+1$.
\end{theorem}

\begin{proof}
We set $c=E(\gamma)$ and $n=\ind(\gamma)$.
Let $\Sigma:(W,\partial W)\to(\Delta^{\leq c},\Delta^{<c})$ be the $S^1$-equivariant extension of a spherical descending manifold of $\gamma$, and $k$ a generator of $\HS{n}(W,\partial W)\cong\Q$. The homology classes $\Sigma_*(k)$ and $\overline\Sigma_*(k)$ are linearly independent in $\HS{n}(\Delta^{\leq c},\Delta^{<c})$, and in particular their sum
\begin{align*}
 h:=\Sigma_*(k)+\overline\Sigma_*(k) \in \HS{n}(\Delta^{\leq c},\Delta^{<c})
\end{align*}
is non-zero. For each $a\in(c,\infty]$, we denote by $\iota^a:(\Delta^{\leq c},\Delta^{<c})\hookrightarrow(\Delta^{<a},\Delta^{<c})$ the inclusion. 
If $a$ is sufficiently close to $c$, then $\iota^a$ is a homotopy equivalence, and in particular $\iota_*^a(h)\neq0$. In contrast, by Lemma~\ref{l:Fet}, we have $\iota^\infty_*(h)=0$, and therefore $\iota^a_*(h)=0$ for all $a\in(c,\infty)$ large enough as well. The transition occurs at a critical value
\begin{align*}
 b
 :=
 \inf\big\{ a\in(c,\infty)\ \big|\ \iota^a_*(h)=0 \big\}
 =
 \max\big\{ a\in(c,\infty)\ \big|\ \iota^a_*(h)\neq0 \big\}.
\end{align*}
We consider the triple of spaces $\Delta^{<c}\subset\Delta^{<b}\subset\Delta^{\leq b}$, and the associated long exact sequence
\[
...
\longrightarrow
\HS{n+1}(\Delta^{\leq b},\Delta^{<b})
\overset{\partial_*}{\longrightarrow}
\HS{n}(\Delta^{<b},\Delta^{<c})
\overset{i_*}{\longrightarrow}
\HS{n}(\Delta^{\leq b},\Delta^{<c})
\longrightarrow
...
\]
where $i_*$ is induced by the inclusion, and $\partial_*$ is the connecting homomorphism. The above definition of the critical value $b$ implies that 
\[\iota_*^b(h)\in\ker(i_*)=\im(\partial_*),\] 
and in particular $\HS{n+1}(\Delta^{\leq b},\Delta^{<b})\neq 0$. This implies that there exists a closed geodesic $\zeta\in\crit^+(E)\cap E^{-1}(b)\cap\Delta$ of index $\ind(\zeta)=n+1$.
\end{proof}

\section{Closed geodesics of Anosov Riemannian manifolds}
\label{s:Anosov_closed_geodesics}

\subsection{Parity of the indices}

The first step towards Theorem~\ref{mt:Klingenberg} consists in establishing the parity of the indices of the closed geodesics of an Anosov Riemannian manifold.

\begin{proposition}\label{p:even_ind}$ $
Every closed geodesic of an Anosov Riemannian manifold has even index.
\end{proposition}

The proof of this proposition is rather straightforward for Anosov Riemannian manifolds of dimension at least 3, but surprisingly requires a non-trivial argument based on the reversibility of the geodesic flow for Anosov Riemannian surfaces. We provide separate proofs for the two cases.

\begin{proof}[Proof of Proposition~\ref{p:even_ind} for $\dim(M)\geq3$]
Let $\gamma\in\crit^+(E)$ be a contractible closed geodesic, $x:=\gamma(0)$, and $v:=\dot\gamma(0)/\|\dot\gamma(0)\|_g$. We consider the lift $\Gamma:S^1\to SM$, $\Gamma(t):=\dot\gamma(t)/\|\dot\gamma(t)\|_g$, which is the corresponding reparametrized orbit of the geodesic flow, i.e.
\[\Gamma(t):=\phi_{t\tau}(v),\qquad\forall t\in\R.\]  
The unit tangent bundle $\proj:SM\to M$ induces a long exact sequence of homotopy groups
\begin{align*}
...
\longrightarrow
\pi_1(S_xM,v)
\longrightarrow
\pi_1(SM,v)
\overset{\proj_*}{\longrightarrow}
\pi_1(M,x)
\longrightarrow
...
\end{align*}
Since $n:=\dim(M)\geq3$, the $(n-1)$-sphere $S_xM$ is simply connected, and the above exact sequence implies that the base projection $\proj$ induces an injective homomorphism
\begin{align*}
\proj_*:\pi_1(SM,v)
\hookrightarrow
\pi_1(M,x).
\end{align*}
This, together with the fact that $\gamma$ is contractible, implies that $\Gamma$ is contractible in the unit tangent bundle. The stable bundle of $\Gamma$ is the restriction $\Es|_{\Gamma}$ of the ambient stable bundle $\Es\subset T(SM)$. Since $\Gamma$ is contractible we infer that $\Es|_{\Gamma}$ is an orientable bundle. Since the closed geodesic $\gamma$ is contractible, its normal bundle $N\gamma$ is orientable as well. Therefore, by Proposition~\ref{p:index_parity}, $\ind(\gamma)$ is even.

Consider now a connected component $U$ of the free loop space $\Lambda$ containing non-contractible loops. Any such $U$ always contains a closed geodesic of even index: any $\gamma_0\in\crit^+(E)\cap U$ of minimal energy $E(\gamma_0)=\min E|_U$, and thus of index $\ind(\gamma_0)=0$. Assume now that there exists another closed geodesic $\gamma_1\in\crit^+(E)\cap U$ of positive index. Since $\dim(M)\geq3$, the homotopic immersed loops $\gamma_0$ and $\gamma_1$ are actually regularly homotopic, namely homotopic through a continuous family of smooth immersions $\gamma_s:S^1\looparrowright M$, $s\in[0,1]$. We lift $\gamma_s$ to a homotopy $\Gamma_s:S^1\to SM$, $\Gamma_s(t)=\dot\gamma_s(t)/\|\dot\gamma_s(t)\|_g$. The stable bundles $\Es|_{\Gamma_0}$ and $\Es|_{\Gamma_1}$ are isomorphic, since they can be interpolated through the family of  bundles $\Es|_{\Gamma_s}$. In particular $\Es|_{\Gamma_0}$ and $\Es|_{\Gamma_1}$ are either both orientable or both unorientable. Analogously, the normal bundles $N\gamma_0$ and $N\gamma_1$ are isomorphic, since they can be interpolated through the family of normal bundles $N\gamma_s$, and therefore $N\gamma_0$ and $N\gamma_1$ are either both orientable or both unorientable. This, together with Proposition~\ref{p:index_parity}, implies that $\ind(\gamma_0)$ and $\ind(\gamma_1)$ have the same parity, and therefore $\ind(\gamma_1)$ is even.
\end{proof}

The case of Anosov Riemannian surfaces requires the following ingredient. The proof that we provide relies on the reversibility of the Riemannian geodesic flow.

\begin{lemma}\label{l:Es_fiber_orientable}
Let $(M,g)$ be an Anosov Riemannian manifold of dimension $2$. For each $x\in M$, the restricted stable bundle $\Es|_{S_xM}$ is an orientable vector bundle.
\end{lemma}

\begin{proof}
We employ some elements of the geometry of the unit tangent bundle of Riemannian surfaces, which can be found in many textbooks, e.g.\ \cite[Sect.~1.11]{Guillarmou:2026aa}. Let $\proj:SM\to M$ be the base projection, and $\kappa:T(TM)\to TM$ the connection map of the Riemannian surface $(M,g)$, defined by $\kappa_v(w)=\nabla_t W|_{t=0}$ for all $v\in TM$ and for all smooth $W:(-\epsilon,\epsilon)\to SM$ such that $W(0)=v$ and $\dot W(0)=w$. Here, $\nabla_t$ denotes the Levi-Civita covariant derivative. 
We fix a point $x\in M$, and an open neighborhood $U\subset M$ of $x$ that is orientable (if $M$ itself is orientable, we can choose $U=M$). We fix an arbitrary orientation on $U$, and denote by $J:TU\to TU$ its complex structure, which rotates every tangent vector by a positive angle of $\pi/2$ in every fiber of $TU$. The maps $d\proj$ and $\kappa$ together provide isomorphisms
\begin{align*}
 \Pi_v:=(d\proj(v),\kappa_v):T_v(SM)\to T_{p(v)}M\oplus\langle v\rangle^\bot,
 \qquad\forall v\in SU,
\end{align*}
where $\langle v\rangle^\bot$ denotes the orthogonal complement of $v$ in the tangent space $T_{p(v)}M$.
The tangent bundle $T(SU)$ can be trivialized by means of three vector fields $X$, $Y$, and $V$, which are defined by the equations
\begin{align*}
\Pi_v(X(v))=(v,0),
\qquad
\Pi_v(Y(v))=(Jv,0),
\qquad
\Pi_v(V(v))=(0,Jv).
\end{align*}
Let $\lambda$ be the canonical contact form of $SM$, already introduced in~\eqref{e:contact_form}. The vector field $X$ is the geodesic vector field, which generates the geodesic flow $\phi_t$ and satisfies $\lambda(X)\equiv1$ and $d\lambda(X,\cdot)\equiv0$. The other two vector fields $Y$ and $V$ span the contact distribution $\ker(\lambda)$ over $SU$. Since $d\phi_t(v)X(v)=X(\phi_t(v))$ and $d\phi_t(v)\ker(\lambda_v)=\ker(\lambda_{\phi_t(v)})$, property (\ref{i:Anosov2}) in the definition of Anosov flow implies that $\Es\oplus\Eu=\ker(\lambda)$. The unit tangent bundle admits the fiberwise involution $\iota:SM\to SM$, $\iota(v)=-v$. The geodesic flow is reversible, meaning that 
\begin{align}
\label{e:reversibility}
\phi_t\circ\iota=\iota\circ\phi_{-t}, 
\end{align}
and $d\iota(v)X(v)=-X(-v)$ at the infinitesimal level. The other two vector fields are transformed by the linearized involution as
\begin{align}
\label{e:iota_Y_V}
 d\iota(v)Y(v)=-Y(-v),
 \qquad
 d\iota(v)V(v)=V(-v).
\end{align}

Fix a unit tangent vector $v_0\in S_xM$. We parametrize the vectors in the fiber $S_xM$ by an angle variable $\theta\in\R$, setting $v_\theta=v_{\theta+2\pi}=e^{\theta J}v_0$. There exist continuous functions $\alpha_s:\R\to\R$ and $\alpha_u:\R\to\R$, which are unique up to an additive multiple of $\pi$, such that 
\begin{align}
\label{e:es}
e_s(\theta)
&
:=\cos(\alpha_s(\theta)) Y(v_\theta)  +\sin(\alpha_s(\theta)) V(v_\theta) \in \Es_{v_\theta},\\
\label{e:eu}
e_u(\theta)
&
:=\cos(\alpha_u(\theta)) Y(v_\theta)  + \sin(\alpha_u(\theta)) V(v_\theta) \in \Eu_{v_\theta}.
\end{align}
Since $e_s(\theta)$ and $e_u(\theta)$ are linearly independent, we infer
\begin{align}
\label{e:alphas_alphau}
 \alpha_s(\theta)-\alpha_u(\theta)\not\equiv0 \mbox{ mod }\pi,\qquad\forall \theta\in\R.
\end{align}
By~\eqref{e:reversibility}, we have
\begin{align*}
d\iota(v)\Es_{v_\theta}=\Eu_{v_{\theta+\pi}}.
\end{align*}
By~\eqref{e:iota_Y_V} and~\eqref{e:es}, we have
\begin{align*}
d\iota(v_\theta)e_s(\theta)
=
-\cos(\alpha_s(\theta)) Y(v_{\theta+\pi})  +\sin(\alpha_s(\theta)) V(v_{\theta+\pi}).
\end{align*}
Since $d\iota(v_\theta)e_s(\theta)\in\Eu_{v_{\theta+\pi}}$, the latter identity together with~\eqref{e:eu} implies that 
\begin{align*}
 \alpha_s(\theta) + \alpha_u(\theta+\pi) \equiv 0 \mbox{ mod }\pi,\qquad\forall\theta\in\R.
\end{align*}
This, together with~\eqref{e:alphas_alphau}, implies
\begin{align*}
\delta(\theta)
:=
\alpha_s(\theta)+\alpha_s(\theta+\pi)
\not\equiv0\mbox{ mod }\pi,\qquad\forall\theta\in\R
\end{align*}
Therefore, there exists an integer $k\in\Z$ such that $\delta(\theta)\in(k\pi,(k+1)\pi)$ for all $\theta\in\R$, and in particular the function $\delta:\R\to\R$ is uniformly bounded. Since $e_s(\theta+2\pi)\in\{e_s(\theta),-e_s(\theta)\}$, there exists $h\in\Z$ such that
\[
\alpha_s(\theta+2\pi)=\alpha_s(\theta)+h\pi,\qquad\forall\theta\in\R.
\]
This implies
\begin{align*}
 \delta(\theta+\pi)
 =
 \alpha_s(\theta+\pi)+\alpha_s(\theta+2\pi)
 =
 \delta(\theta) + h\pi,
 \qquad\forall\theta\in\R.
\end{align*}
Since the function $\delta$ is uniformly bounded, we infer that $h=0$, and therefore 
\[e_s(\theta+2\pi)=e_s(\theta),\qquad\forall\theta\in\R.\]
Namely, we have a nowhere vanishing section $e_s:\R/2\pi\Z\to\Es|_{S_xM}$, $e_s(\theta)\in\Es_{v_\theta}$ for all $\theta\in\R/2\pi\Z$. Such $e_s$ is a trivialization of the line bundle $\Es|_{S_xM}$, and we conclude that $\Es|_{S_xM}$ is orientable.
\end{proof}

\begin{proof}[Proof of Proposition~\ref{p:even_ind} for $\dim(M)=2$]
For each $x\in M$ and $v\in S_xM$, we consider once again the homotopy long exact sequence of the unit tangent bundle
\begin{align}\label{e:exact_sequence_SM}
...
\longrightarrow
\pi_1(S_xM,v)
\longrightarrow
\pi_1(SM,v)
\overset{\proj_*}{\longrightarrow}
\pi_1(M,x)
\longrightarrow
...
\end{align}
Let $\gamma\in\crit^+(E)$ be a contractible closed geodesic, and $\Gamma:S^1\to SM$, $\Gamma(t)=\dot\gamma(t)/\|\dot\gamma(t)\|_g$ its lift to the unit tangent bundle, which is the corresponding re\-pa\-ram\-e\-trized periodic orbit of the geodesic flow.
Since $\gamma$ is contractible, the above long exact sequence for $x:=\gamma(0)$ and $v:=\Gamma(0)$ implies that $\Gamma$ is homotopic to a loop $\Psi$ contained in the fiber circle $S_xM$. By Lemma~\ref{l:Es_fiber_orientable}, the restricted stable bundle $\Es|_\Psi$ is orientable. Since $\Gamma$ and $\Psi$ are homotopic, $\Es|_\Gamma$ and $\Es|_\Psi$ are isomorphic, and therefore $\Es|_\Gamma$ is an orientable vector bundle. Moreover, since $\gamma$ is contractible, its normal bundle $N\gamma$ is orientable. By Proposition~\ref{p:index_parity}, we infer that $\ind(\gamma)$ is even.

Consider now a connected component $U$ of the free loop space $\Lambda$ containing non-contractible loops, and a closed geodesic $\gamma\in\crit^+(E)\cap U$ of minimal energy $E(\gamma)=\min E|_U$, which must have index $\ind(\gamma)=0$. Assume that there exists another closed geodesic $\zeta\in\crit^+(E)\cap U$ of positive index. There exists a continuous path $\nu:[0,1]\to M$ joining $\nu(0)=\gamma(0)=:x$ and $\nu(1)=\zeta(0)=:y$ such that $[\gamma]=[\nu*\zeta*\overline\nu]$ in $\pi_1(M,x)$. Here, $*$ denotes the concatenation of paths, $\overline\nu=\nu(1-\cdot)$ denotes the path $\nu$ with reversed orientation, and the resulting loop $\nu*\zeta*\overline\nu:[0,3]\to M$ is implicitly reparametrized on $S^1=\R/\Z$ to represent an element of the fundamental group $\pi_1(M,x)$. Let $\alpha:[0,1]\to M$ be a smooth immersed path with endpoints $\alpha(0)=x$ and $\alpha(1)=y$, homotopic to $\nu$ through continuous paths joining $x$ and $y$, and with tangent vectors $\dot\alpha(0)=\dot\gamma(0)$ and $\dot\alpha(1)=\dot\zeta(0)$. Analogously, let $\beta:[0,1]\to M$ be a smooth immersed path with endpoints $\beta(0)=y$ and $\beta(1)=x$, homotopic to $\overline\nu$ through continuous paths joining $y$ and $x$, and with tangent vectors $\dot\beta(0)=\dot\zeta(0)$ and $\dot\beta(1)=\dot\gamma(0)$.
Notice that the concatenation $\sigma:=\alpha*\zeta*\beta$ is a smooth immersed loop such that $\dot\sigma(0)=\dot\gamma(0)$, and $[\sigma]=[\gamma]$ in $\pi_1(M,x)$. 
We consider the lifts 
\[
\Gamma:=\frac{\dot\gamma}{\|\dot\gamma\|_g},
\qquad 
A:=\frac{\dot\alpha}{\|\dot\alpha\|_g},
\qquad
B:=\frac{\dot\beta}{\|\dot\beta\|_g},
\qquad 
Z:=\frac{\dot\zeta}{\|\dot\zeta\|_g},
\]
and the concatenation 
$\Sigma:=A*Z*B$. We reparametrize both $\sigma$ and $\Sigma$ on $S^1$, so that 
\begin{align*}
\Sigma=\frac{\dot\sigma}{\|\dot\sigma\|_g}.
\end{align*}
Notice that $v:=\Sigma(0)=\Gamma(0)$. Since $\proj_*[\Sigma]=\proj_*[\Gamma]$ in $\pi_1(M,x)$, the exact sequence \eqref{e:exact_sequence_SM} implies that the concatenated loop $\Sigma*\overline\Gamma$ is homotopic to a loop $\Psi$ contained in the fiber $S_xM$. Therefore, the restricted stable bundles $\Es|_{\Sigma*\overline\Gamma}$ and $\Es|_\Psi$ are isomorphic. By Lemma~\ref{l:Es_fiber_orientable}, the vector bundle $\Es|_\Psi$ is orientable. Therefore $\Es|_{\Sigma*\overline\Gamma}$ is orientable as well, and thus a trivial vector bundle (since a vector bundle over a circle is orientable if and only if it is trivial). We infer that $\Es|_{\Sigma}$ and $\Es|_{\Gamma}$ are isomorphic vector bundles.
Since $\beta*\alpha=\proj(B*A)$ is a contractible loop, we have $\proj_*[B*A]=1$ in $\pi_1(M,y)$. Once again, the exact sequence \eqref{e:exact_sequence_SM} implies that the loop $B*A$ is homotopic to a loop $\Phi$ contained in the fiber $S_yM$. By Lemma~\ref{l:Es_fiber_orientable}, the vector bundle $\Es|_\Phi$ is orientable. Therefore,  $\Es|_{B*A}$ is orientable as well, and thus a trivial bundle, and we infer that $\Es|_{\Sigma}$ and $\Es|_{Z}$ are isomorphic. Overall, we showed that $\Es|_{Z}$ and $\Es|_{\Gamma}$ are isomorphic, and thus either both orientable or both unorientable. Since $\gamma$ and $\zeta$ are homotopic, the restricted tangent bundles $TM|_{\gamma}$ and $TM|_{\zeta}$ are isomorphic, and therefore either both orientable or both unorientable. Moreover, such bundles split as 
\[
TM|_{\gamma}=N\gamma\oplus\langle\dot\gamma\rangle,
\qquad
TM|_{\zeta}=N\zeta\oplus\langle\dot\zeta\rangle.
\]
Clearly, the line bundles $\langle\dot\gamma\rangle$ and $\langle\dot\zeta\rangle$ are oriented by the tangent vectors $\dot\gamma$ and $\dot\zeta$ respectively. Therefore, the four vector bundles $TM|_{\gamma}$, $N\gamma$, $TM|_{\zeta}$, $N\zeta$ are either all orientable or all unorientable. We can finally invoke Proposition~\ref{p:index_parity}, and infer that $\ind(\gamma)$ and $\ind(\zeta)$ have the same parity, and therefore $\ind(\zeta)$ is even.
\end{proof}

In a bumpy closed Riemannian manifold, Morse theory implies that, if the $S^1$-equivariant homology $\HS{n}(\Lambda^{<b},\Lambda^{<a})$ is non-trivial for some $b>a>0$, there exists a closed geodesic $\gamma$ such that $E(\gamma)\in[a,b)$ and $\ind(\gamma)=n$. This, together with Proposition~\ref{p:even_ind}, implies that the energy functional of an Anosov Riemannian manifold is \emph{perfect} for the $S^1$-equivariant homology relative to $\Lambda^0$ with rational coefficients.

\begin{corollary}[$S^1$-equivariant perfectness]\label{c:perfectness}
On any Anosov Riemannian manifold, for each $0<a<b<c$, the $S^1$-equivariant homology long exact sequence of the triple $\Lambda^{<a}\subset\Lambda^{<b}\subset\Lambda^{<c}$ reduces to a short exact sequence
\begin{align*}
 0
 \longrightarrow
 \HS*(\Lambda^{<b},\Lambda^{<a})
 \longrightarrow
 \HS*(\Lambda^{<c},\Lambda^{<a})
 \longrightarrow
 \HS*(\Lambda^{<c},\Lambda^{<b})
 \longrightarrow
 0.
\end{align*}
\end{corollary}

\subsection{Closed geodesics of index zero}
As a first step towards Theorem~\ref{mt:Klingenberg}, we describe the closed geodesics of index zero in an Anosov Riemannian manifold.

\begin{lemma}\label{l:index_zero}
On any Anosov Riemannian manifold, the following assertions hold.
\begin{enumerate}[$(i)$]
\setlength{\itemsep}{3pt}
\item\label{i:index_zero_contractible} There exists no contractible closed geodesic $\gamma$ of index $\ind(\gamma)=0$.

\item\label{i:index_zero_non_contractible} On every connected component of non-contractible loops $U\subset\Lambda$, there exists a unique critical circle $S^1\cdot\gamma$ of closed geodesics of index $\ind(\gamma)=0$.

\end{enumerate}
\end{lemma}

\begin{proof}$ $
\begin{enumerate}[(i)]
\setlength{\itemsep}{3pt}

\item\label{i:zeroindex_1} Assume by contradiction that there exists a contractible closed geodesic $\gamma$ of index $\ind(\gamma)=0$. We consider the min-max value
\begin{align*}
 c:=\inf_{\Gamma}\max_{s\in[0,1]} E(\Gamma(s)),
\end{align*}
where the infimum ranges over the space of continuous paths $\Gamma:[0,1]\to\Lambda$ such that $\Gamma(0)=\gamma$ and $\Gamma(1)\in\Lambda^0$. By Morse theory, $c$ is a critical value of $E$.
We claim that there exists a critical circle in $\crit(E)\cap E^{-1}(c)$ of index 1. This contradicts Proposition~\ref{p:even_ind}.

The proof of the claim is very standard, and we present it here for the reader's convenience. Since $\gamma$ is non-degenerate and has index $\ind(\gamma)=0$, the critical circle $S^1\cdot\gamma$ is a local minimizer of $E$. The Morse-Bott lemma \eqref{e:Morse_Bott} provides an arbitrarily small neighborhood $U\subset\Lambda\setminus\Lambda^0$ of $S^1\cdot\gamma$ such that $\inf E|_{\partial U}>E(\gamma)$. Any $\Gamma$ as above must exit $U$ in order to reach $\Lambda^0$, and in particular $c>\inf E|_{\partial U}>E(\gamma)$. 

Assume by contradiction that no critical circle in $E^{-1}(c)$ has index 1. Therefore, we have $\crit(E)\cap E^{-1}(c)=K_0\cup K_2$, where each $\zeta\in K_0$ has index $\ind(\zeta)=0$, whereas any $\zeta\in K_2$ has index $\ind(\zeta)\geq2.$
By Morse-Bott Lemma \eqref{e:Morse_Bott}, there exists an open neighborhood $Z$ of $K_0$ such that
\[b:= \inf E|_{\partial Z}>c,\]
whereas each $S^1\cdot\zeta\subset K_2$ has an open neighborhood $W(S^1\cdot\zeta)$ such that the intersection $W(S^1\cdot\zeta)\cap\Lambda^{<c}$ is path-connected. We set
\begin{align*}
 W:=\bigcup_{S^1\cdot\zeta\subset K_2} W(S^1\cdot\zeta)
\end{align*}
Let $\Gamma$ be a path as above that is almost optimal, so that $a:=\max E\circ\Gamma<b$ and $\crit^+(E)\cap E^{-1}[c,a]\subset Z\cup W$. By pushing $\Gamma$ with the anti-gradient flow $\Psi_s$ of the energy functional $E$, namely replacing $\Gamma$ with $\Psi_s\circ\Gamma$ for a sufficiently large $s>0$, we can assume that the image of $\Gamma$ is contained in $\Lambda^{<c}\cup Z\cup W$. Actually, since $a<b$, we infer that $\Gamma$ does not enter $Z$. Moreover, every portion $\Gamma|_{[a,b]}$ contained in some $W(S^1\cdot\zeta)$ and with endpoints $\Gamma(a),\Gamma(b)\in\Lambda^{<c}$ can be replaced with a new segment entirely contained in $\Lambda^{<c}$ and having the same endpoints $\Gamma(a)$ and $\Gamma(b)$. Namely, in the above min-max scheme, we can choose a path $\Gamma$ that is entirely contained in $\Lambda^{<c}$, which contradicts the very definition of $c$.

\item Let $U$ be a connected component of the free loop space $\Lambda$ containing non-contractible loops. On $U$, the energy functional achieves its minimum $c:=\min E|_U$. Any closed geodesic $\gamma\in\crit(E)\cap E^{-1}(c)\cap U$ has index $\ind(\gamma)=0$. Assume by contradiction that there exists another closed geodesic $\zeta\in\crit^+(E)\cap U\setminus S^1\cdot\gamma$ of index $\ind(\zeta)=0$. We consider the min-max value
\begin{align*}
 b:=\inf_{\Gamma}\max_{s\in[0,1]} E(\Gamma(s)),
\end{align*}
where the infimum ranges over the space of continuous paths $\Gamma:[0,1]\to\Lambda$ such that $\Gamma(0)=\gamma$ and $\Gamma(1)=\zeta$. Arguing as in point (\ref{i:zeroindex_1}), we infer that $b$ is a critical value of $E$ and there exists a critical circle in $\crit(E)\cap E^{-1}(b)$ of index 1, which contradicts Proposition~\ref{p:even_ind}.
\qedhere

\end{enumerate}
\end{proof}

Lemma~\ref{l:index_zero} has the following immediate consequence on the fundamental group of an Anosov Riemannian manifold.

\begin{corollary}\label{c:fundamental_group}
The fundamental group of an Anosov Riemannian manifold has no element of finite order. Namely, for any non-contractible loop $\gamma\in\Lambda$ and for any integer $m\geq2$, the $m$-th iterate $\gamma^m$ is non-contractible as well.
\end{corollary}

\begin{proof}
Let $U$ be a connected component of the free loop space $\Lambda$ consisting of non-contractible loops, and $\gamma\in\crit^+(E)\cap U$ a closed geodesic such that $E(\gamma)=\min E|_U$. Being a minimizer of the energy in its connected component, $\gamma$ has index $\ind(\gamma)=0$. Since the closed Riemannian manifold is assumed to be Anosov, $\gamma$ is hyperbolic, and $\ind(\gamma^m)=m\,\ind(\gamma)=0$ for all integers $m\geq1$. Lemma~\ref{l:index_zero}(\ref{i:index_zero_contractible}) implies that every iterate $\gamma^m$ is non-contractible.
\end{proof}

\subsection{Proof of Klingenberg's Theorem~\ref{mt:Klingenberg}}
We address the three points of Theorem~\ref{mt:Klingenberg} separately, and we rewrite them for the reader's convenience.\vspace{10pt}

\noindent\textbf{Theorem~\ref{mt:Klingenberg}(\ref{i:Klingenberg_1}).} \emph{On any Anosov Riemannian manifold, no closed geodesic is contractible.}

\begin{proof}
Assume by contradiction that the Anosov Riemannian manifold $(M,g)$ admits a contractible closed geodesic. Let $\gamma$ be any such closed geodesic of minimal energy among all contractible closed geodesics. This implies that $\gamma$ cannot be the multiple iterate of a contractible closed geodesic. By Corollary~\ref{c:fundamental_group}, $\gamma$ cannot be the multiple iterate of a non-contractible closed geodesic either. Therefore $\gamma$ is primitive. 
By Lemma~\ref{l:index_zero}(\ref{i:index_zero_contractible}), $\ind(\gamma)>0$.
By Lemma~\ref{l:shortest_contractible}, $\gamma$ is spherical. 
By Theorem~\ref{t:Fet}, there exists a contractible closed geodesic $\zeta\in\crit^+(E)$ of index $\ind(\zeta)=\ind(\gamma)+1$. Therefore, either $\ind(\gamma)$ or $\ind(\zeta)$ must be odd, which contradicts Proposition \ref{p:even_ind}.
\end{proof}

The previous theorem has the following consequences on the topology of Anosov Riemannian manifolds and of their free loop spaces.

\begin{corollary}\label{c:homotopy_M}
All higher homotopy groups of any Anosov Riemannian manifold $M$ vanish, i.e. 
\[
\pi_{n}(M)=0,\qquad\forall n\geq2.
\]
\end{corollary}

\begin{proof}
Assume by contradiction that $\pi_n(M)\neq0$ for some $n\geq2$. Namely, there exists a continuous map $\sigma_0:S^n\to M$ that is not homotopic to a constant map. This, together with Corollary~\ref{c:spherical_maps}, implies that there exists a continuous map \[\sigma:(B^{n-1},\partial B^{n-1})\to(\Lambda,\Lambda^0)\] that is not homotopic as a map of pairs to a continuous map whose image is entirely contained in $\Lambda^0$. Notice that $\sigma$ takes values in the connected component $\Delta$ of $\Lambda$ consisting of the contractible loops.

We now run the classical min-max argument of Lusternik-Fet \cite{Lyusternik:1951aa}, originally introduced by Birkhoff \cite{Birkhoff:1917aa} in the special case of $S^2$.
The min-max value
\begin{align*}
 c:=\inf_{\nu}\max_{x\in B^{n-1}} E(\nu(x)),
\end{align*}
where the infimum ranges over all continuous maps $\nu:(B^{n-1},\partial B^{n-1})\to(\Delta,\Lambda^0)$ homotopic as maps of pairs to $\sigma$, is a critical value of the energy functional $E|_\Delta$.  
We claim that 
\[c\geq\inj(M,g)^2.\] 
Indeed, assume by contradiction that $c<\inj(M,g)^2$, so that there exists $\nu$ as above such that $\max E\circ\nu<\inj(M,g)^2$. Namely, each loop in the image of $\nu$ has length smaller than the injectivity radius $\inj(M,g)$. We can employ the Riemannian exponential map to build a homotopy $h_s:(B^{n-1},\partial B^{n-1})\to(\Delta,\Lambda^0)$ as
\begin{align*}
 h_s(x)(t)=\exp_{\nu(x)(0)}\big( (1-s)\exp_{\nu(x)(0)}^{-1}(\nu(x)(t)) \big).
\end{align*}
For $s=0$, we have $h_0=\nu$, whereas for $s=1$ the map $h_1$ has image contained in the space of constant loops $\Lambda^0$. But this is not possible, since $\sigma$ is not homotopic as a map of pairs to a map with image in $\Lambda^0$.

Since $c>0$, we have found a contractible closed geodesic $\gamma\in\crit^+(E)\cap\Delta$, which contradicts Theorem~\ref{mt:Klingenberg}(\ref{i:Klingenberg_1}).
\end{proof}

\begin{corollary}\label{c:homotopy_Lambda}
All higher homotopy groups of the free loop space $\Lambda$ of any Anosov Riemannian manifold vanish, i.e.
\[
\pi_{n}(\Lambda,\gamma)=0,\qquad\forall \gamma\in\Lambda,\ n\geq2.
\]
\end{corollary}

\begin{proof}
We fix $\gamma\in\Lambda$ and $y=\gamma(0)$, which will serve as basepoints for the corresponding spaces. The continuous evaluation map $\ev:\Lambda\to M$, $\ev(\zeta)=\zeta(0)$ has the structure of a fibration, the fiber $\ev^{-1}(x)$ being the based loop space $\Omega_x$. We have an associated long exact sequence of homotopy groups
\begin{align*}
 ...
 \overset{\partial_*}{\longrightarrow}
 \pi_n(\Omega_y,\gamma)
 \longrightarrow
 \pi_n(\Lambda,\gamma)
 \overset{\ev_*\ }{\longrightarrow}
 \pi_n(M,y)
 \overset{\partial_*}{\longrightarrow}
 \pi_{n-1}(\Omega_y,\gamma)
 \longrightarrow
 ...
\end{align*}
As we already recalled in Remark~\ref{r:based_loop_space}, we have isomorphisms of homotopy groups $\pi_n(\Omega_y,\gamma)\cong\pi_{n+1}(M,y)$ in all degrees $n\geq0$. By Corollary~\ref{c:homotopy_M}, we have $\pi_{n}(M,y)=0$ in all degrees $n\geq2$, and therefore $\pi_n(\Omega_y,\gamma)=0$ in all degrees $n\geq1$. 
This, together with the above long exact sequence, implies that $\pi_n(\Lambda,\gamma)=0$ in all degrees $n\geq2$.
\end{proof}

Let $\Delta$ be the connected component of $\Lambda$ consisting of the contractible loops  of an Anosov Riemannian manifold $M$. Theorem~\ref{mt:Klingenberg}(\ref{i:Klingenberg_1}) implies that the inclusion of the constant loops $\Lambda^0\hookrightarrow\Delta$ is a homotopy equivalence, whose homotopy inverse can be built by means of the anti-gradient flow of $E$. In particular, the fundamental group of $\Delta$ is isomorphic to that of $M$. The situation is different for the  connected components of the free loop space containing non-contractible loops.

\begin{lemma}\label{l:pi_1_Lambda}
For each non-contractible loop $\gamma\in\Lambda$ in an Anosov Riemannian manifold, we have
\begin{align*}
 \pi_1(\Lambda,\gamma)\cong\Z.
\end{align*}
\end{lemma}

\begin{proof}
Fix a non-contractible loop $\gamma\in\Lambda$ in an Anosov Riemannian manifold $M$.
The $S^1$ action on $\Lambda$ provides a loop $\Gamma:S^1\to\Lambda$ given by
\begin{align*}
 \Gamma(s)=s\cdot\gamma.
\end{align*}
We claim that $[\Gamma]$ is an element of infinite order in the fundamental group $\pi_1(\Lambda,\gamma)$. Indeed, assume by contradiction that $[\Gamma]^m=1$ for some integer $m\geq1$. This means that we have a homotopy $h_r:S^1\to\Lambda$ such that $h_0(s)=\Gamma(ms)$ for all $s\in[0,1]$, $h_r(0)=\gamma$ for all $r\in[0,1]$, and $h_1(s)=\gamma$ for all $s\in[0,1]$. We obtain an associated homotopy $k_r:S^1\to M$, $k_r(s)=h_r(s)(0)$, which satisfies $k_0=\gamma^m$ and $k_1\equiv\gamma(0)$ (see Figure~\ref{f:homotopy}). This implies that $\gamma^m$ is contractible, which contradicts Corollary~\ref{c:fundamental_group}.

\begin{figure}
\begin{center}
\includegraphics{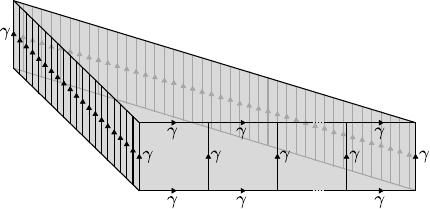}
\caption{The homotopy $h_r$ visualized as a map from a polytope to $M$. All edges with arrows are mapped to $\gamma$. The upper face is the homotopy $k_r$ that contracts the loop $\gamma^m$.}
\label{f:homotopy}
\end{center}
\end{figure}

Since the fundamental group $\pi_1(\Lambda,\gamma)$ does not depend on the specific choice of $\gamma$ within a connected component $U$ of $\Lambda$, we can assume without loss of generality that $E(\gamma)=\min E|_U$. We claim that the relative homotopy $\pi_1(\Lambda,S^1\cdot\gamma,\gamma)$ is trivial. Indeed, assume by contradiction that this is not the case, so that there exists a continuous path $\Gamma:([0,1],\{0,1\},\{0\})\to(\Lambda,S^1\cdot\gamma,\gamma)$ that is not homotopic as a map of triples to the constant path at $\gamma$. We consider the min-max value
\begin{align*}
c:=\inf_{\Psi}\max_{s\in[0,1]} E(\Psi(s)),
\end{align*}
where the infimum ranges over all continuous paths 
\[\Psi:([0,1],\{0,1\},\{0\})\to(\Lambda,S^1\cdot\gamma,\gamma)\] 
homotopic as a map of triples to $\Gamma$. 
The min-max $c$ is a critical value of $E|_U$ larger than $E(\gamma)$. Arguing as in the proof of Lemma~\ref{l:index_zero}, we infer that there exists a closed geodesic $\zeta\in\crit^+(E)\cap E^{-1}(c) \cap U$ of index $\ind(\zeta)=1$. This contradicts Proposition~\ref{p:even_ind}.

Since $\pi_1(\Lambda,S^1\cdot\gamma,\gamma)$ is trivial, the inclusion induces a surjective homomorphism 
\begin{align}
\label{e:surjective_pi_1}
\pi_1(S^1\cdot\gamma,\gamma)\rightarrow\pi_1(\Lambda,\gamma).
\end{align}
Since $\pi_1(S^1\cdot\gamma,\gamma)\cong\Z$ and $\pi_1(\Lambda,\gamma)$ contains the element $[\Gamma]$ of infinite order, we infer that the surjective homomorphism \eqref{e:surjective_pi_1} is actually an isomorphism.
\end{proof}

\begin{corollary}\label{c:U_circle}
In any Anosov Riemannian manifold, any connected component of the free loop space $\Lambda$ containing non-contractible loops is homotopy equivalent to~$S^1$.
\end{corollary}

\begin{proof}
Let $U$ be a connected component of the free loop space $\Lambda$ containing non-contractible loops.
By Corollary~\ref{c:homotopy_Lambda} and Lemma~\ref{l:pi_1_Lambda}, the homotopy groups of $U$ are given by
\begin{align*}
 \pi_n(U)
 \cong
\begin{cases}
\Z,& \mbox{if }n=1,\\
0,& \mbox{if }n\geq2.
\end{cases}
\end{align*}
Since the free loop space $\Lambda$ is a Hilbert manifold, it has the homotopy type of a CW-complex, and we can apply Whitehead theorem \cite[Th.~4.5]{Hatcher:2002aa}: since any generator $\Gamma:S^1\to U$ of the fundamental group $\pi_1(U)$ induces an isomorphism of all homotopy groups 
\[\Gamma_*:\pi_*(S^1)\overset{\cong}{\longrightarrow}\pi_*(U),\]
the map $\Gamma$ is a homotopy equivalence.
\end{proof}

\noindent\textbf{Theorem~\ref{mt:Klingenberg}(\ref{i:Klingenberg_2}).} \emph{On any Anosov Riemannian manifold, every homotopy class of non-contractible loops contains a unique closed geodesic. Namely, on any connected component $U$ of the free loop space containing non-contractible loops, the set of critical points of the energy consists of a unique critical circle 
\[\crit(E)\cap U=S^1\cdot\gamma,\]
and $E(\gamma)=\min E|_U$.}

\begin{proof}
Assume by contradiction that there exists a connected component $W\subset\Lambda$ of non-contractible loops containing more than one critical circle of the energy functional $E$. By Lemma~\ref{l:index_zero}(\ref{i:index_zero_non_contractible}), there must exist a critical circle $S^1\cdot\zeta\subset\crit^+(E)\cap W$ of positive index $\ind(\zeta)>0$. Let $m=\mul(\zeta)$, so that $\zeta=\gamma^m$ for some primitive closed geodesic $\gamma\in\crit^+(E)$ of energy $c:=E(\gamma)$. Let $U$ be the connected component of $\Lambda$ containing $\gamma$. 
Since the closed Riemannian manifold is Anosov, and in particular all the closed geodesics are hyperbolic, we have $n:=\ind(\gamma)=\tfrac1m\ind(\zeta)>0$. Let $\epsilon>0$ be small enough so that $(c,c+\epsilon)$ does not contain critical values of the energy functional $E|_U$. As explained in Section~\ref{ss:descending_manifolds}, the primitive closed geodesic $\gamma$ produces non-trivial $S^1$-equivariant homology 
\[
\HS{n}(U^{<c+\epsilon},U^{<c})
\cong
\HS{n}(U^{\leq c},U^{<c})
\neq
0.
\]
Since $E$ is \emph{perfect} for the relative $S^1$-equivariant homology with rational coefficients (Corollary~\ref{c:perfectness}), and since $U^{<a}=\varnothing$ for $a=\min E|_U>0$, we infer that 
\begin{align*}
\HS{n}(U)\neq0.
\end{align*}
The $S^1$-equivariant homology is related to the ordinary homology via the classical Gysin long exact sequence \cite[page 438]{Hatcher:2002aa}
\begin{align*}
 ...
 \longrightarrow
 \HS{*+2}(U)
 \overset{\smallfrown e}{\longrightarrow}
 \HS{*}(U)
 \longrightarrow
 H_{*+1}(U)
 \longrightarrow
 \HS{*+1}(U)
 \overset{\smallfrown e}{\longrightarrow}
 ...
\end{align*}
where $\smallfrown e$ is the homomorphism given by the cap product with the Euler cohomology class $e\in H^2_{S^1}(U)$. Since all the closed geodesics have even index (Proposition~\ref{p:even_ind}), we have
\[\HS{1}(U)=0.\] 
The Gysin sequence implies that the homomorphism $\HS{0}(U)\to H_{1}(U)$ is surjective. By Corollary~\ref{c:U_circle}, we have 
\[H_{1}(U)\cong H_1(S^1)\cong\Q.\] 
Since $\HS{0}(U)\cong\Q$ as well, we infer that the homomorphism $\HS{0}(U)\to H_{1}(U)$ is an isomorphism. Therefore, the Gysin sequence implies that the homomorphism 
\[\smallfrown e:\HS2(U)\to\HS0(U)\] 
vanishes. Let $p$ be the minimal positive integer such that $\HS{2p}(U)\neq0$, and notice that $2p\leq n$. We infer that the homomorphism 
\[\smallfrown e:\HS{2p}(U)\to\HS{2p-2}(U)\] 
vanishes, but $\HS{2p}(U)$ is non-trivial. The Gysin sequence implies that the homomorphism 
\[H_{2p}(U)\to\HS{2p}(U)\] 
is surjective, and in particular $H_{2p}(U)$ is non-trivial. However, Corollary~\ref{c:U_circle} asserts that $U$ is homotopy equivalent to $S^1$, and therefore $H_q(U)$ is trivial in all degrees $q\geq2$. This gives a contradiction.
\end{proof}

Theorem~\ref{mt:Klingenberg}(\ref{i:Klingenberg_1}-\ref{i:Klingenberg_2}) readily implies the following preliminary statement on the absence of conjugate points. We say that a closed geodesic $\gamma\in\crit^+(E)$ is without conjugate points when, for each $t>0$ (including values $t>1$), the points $\gamma(0)$ and $\gamma(t)$ are not conjugate along $\gamma|_{[0,t]}$.

\begin{lemma}
\label{l:no_conjugate_points_closed}
On any Anosov Riemannian manifold, no closed geodesic has conjugate points.
\end{lemma}

\begin{proof}
If a closed geodesic $\gamma\in\crit^+(E)$ has conjugate points, the lower bound \eqref{e:index_bound_conjugate_points} implies that $\ind(\gamma^m)>0$ for some integer $m\geq1$, and therefore $\ind(\gamma)=\tfrac1m\ind(\gamma^m)>0$ as well. But on any Anosov Riemannian manifold, Theorem~\ref{mt:Klingenberg}(\ref{i:Klingenberg_1}-\ref{i:Klingenberg_2}) implies that any closed geodesic $\gamma\in\crit^+(E)$ is a global minimizer of the energy functional $E$ in its connected component of the free loop space, and in particular $\ind(\gamma)=0$.
\end{proof}

Theorem~\ref{mt:Klingenberg}(\ref{i:Klingenberg_3}) is a consequence of Lemma~\ref{l:no_conjugate_points_closed} together with some fundamental properties of general Anosov flows.\vspace{10pt}

\noindent\textbf{Theorem~\ref{mt:Klingenberg}(\ref{i:Klingenberg_3}).} 
\emph{On any Anosov Riemannian manifold, there are no conjugate points.}

\begin{proof}
Let $(M,g)$ be an Anosov Riemannian manifold. We denote by $P\subset SM$ the invariant subset of periodic orbits of its geodesic flow $\phi_t:SM\to SM$, i.e.
\begin{align*}
 P
 =
 \bigcup_{t>0} \fix(\phi_t).
\end{align*}
By Anosov closing lemma \cite[Cor.~11.9]{Guillarmou:2026aa}, $P$ is dense in the non-wandering set of $\phi_t$.
Since $\phi_t$ preserves the canonical contact form $\lambda$, in particular it preserves the volume form $\lambda\wedge (d\lambda)^{n-1}$ on $SM$, where $n=\dim(M)$. Being a volume-preserving Anosov flow, $\phi_t$ is transitive \cite[Prop.~11.15]{Guillarmou:2026aa}, and in particular its non-wandering set is the whole $SM$. Therefore, $P$ is a dense subset of $SM$.

Assume by contradiction that there exist two points $x=\gamma(0)$ and $y=\gamma(\tau)$ that are conjugate along a geodesic segment $\gamma:[0,\tau]\to M$ parametrized with unit speed $\|\dot\gamma\|_g\equiv1$. Let $v_n\in P$ be a sequence converging to $\dot\gamma(0)$ as $n\to\infty$, and $\gamma_n:\R\to M$ the corresponding closed geodesic such that $\dot\gamma_n(0)=v_n$. For $n$ large enough, there exists $\tau_n\in(\tau-1,\tau+1)$ such that the points $\gamma_n(0)$ and $\gamma_n(\tau_n)$ are conjugate along $\gamma_n|_{[0,\tau_n]}$. This contradicts Lemma~\ref{l:no_conjugate_points_closed}.
\end{proof}

\bibliography{biblio}

\vspace{10pt}

\end{document}